\documentclass [12pt, reqno]{amsart} 
\usepackage{times}
\usepackage{xcolor, amssymb,latexsym,amsmath,amsfonts, eucal}
\usepackage[dvipsnames,svgnames,table]{xcolor}
\usepackage{enumitem}

\allowdisplaybreaks[4] 

\newcommand\hlight[1]{\tikz[overlay, remember picture,baseline=-\the\dimexpr\fontdimen22\textfont2\relax]\node[rectangle,fill=white!50,rounded corners,fill opacity = 0.2,draw,thick,text opacity =1] {$#1$};}

\usepackage[bookmarksnumbered, colorlinks, linkcolor=red, citecolor=blue, plainpages]{hyperref}

\newtheorem{thm}{Theorem}[section]

\newtheorem{lem}[thm]{Lemma}
\newtheorem{prop}[thm]{Proposition}
\theoremstyle{definition}
\newtheorem{defn}[thm]{Definition}
\newtheorem{exmp}[thm]{Example}
\theoremstyle{remark}
\newtheorem{rem}[thm]{Remark}

 \DeclareMathOperator{\RE}{Re}
 \DeclareMathOperator{\IM}{Im}

 \newcommand{\M}{\mathcal{M}}

 \newcommand{\abs}[1]{\left\vert#1\right\vert}

 \newcommand{\norm}[1]{\left\Vert#1\right\Vert}

\begin{document}

\title[Noncommutative Maximal Ergodic Theorems]{Noncommutative Maximal Ergodic Theorems for Modulated $(C,\alpha)$-Averages with Operator-Valued Weights}

\author[Chattopadhyay]{Arup Chattopadhyay$^{1}$}
\address{\hskip-\parindent
$1.$ Arup Chattopadhyay, Department of Mathematics, IIT Guwahati, Guwahati-781039, India.}

\email{arupchatt@iitg.ac.in; 2003arupchattopadhyay@gmail.com}

\author[De]{Debabrata De$^{2\ast}$}
\address{\hskip-\parindent
$2.$ Debabrata De, Department of Mathematics, IIT Guwahati, Guwahati-781039, India.}
\email{debabratade.ju@gmail.com; debabratade.ju@iitg.ac.in}

\thanks{${\ast}$ Corresponding author.}

\author[Shaw]{Disha Shaw$^{3}$}
\address{\hskip-\parindent
$3.$ Disha Shaw, Department of Mathematics, IIT Guwahati, Guwahati-781039, India.}

\email{s.disha@iitg.ac.in; dishashaw2565@gmail.com}

\subjclass[2020]{47A35 $($Primary$)$; 
 46L52, 46L10 $($Secondary$)$ }

\keywords{Maximal ergodic inequality, Vector-valued weighted $(C,\alpha)$-averages, Operator-valued Besicovitch sequences, Rota theorem, Noncommutative $L_p$-spaces }

\begin{abstract}
In this paper, we study modulated $(C,\alpha)$-ergodic averages with operator-valued weights and their subsequential versions in the setting of semifinite von Neumann algebras. Our framework combines $(C,\alpha)$-type modulation with operator-valued weights and extends both scalar-weighted and operator-valued weighted ergodic theories. We establish mean ergodic theorems and maximal inequalities for these averages on noncommutative $L_p$-spaces. We then study bilateral almost uniform and almost uniform convergence in noncommutative Orlicz spaces, including convergence along subsequences of density one when $\alpha\geq 1$. A key ingredient in our approach is the introduction of the $p$-convexity condition on the underlying Orlicz function. This condition allows us to transfer appropriate maximal estimates from noncommutative $L_p$-spaces to establish bilateral uniform equicontinuity in measure at zero on the corresponding noncommutative $p$-convex Orlicz spaces. Together with the noncommutative Banach principle, this yields the desired pointwise convergence results.  As an additional consequence of our $p$-convexity argument, we extend Rota theorem to a larger class of noncommutative Orlicz spaces associated with the $p$-convex Orlicz functions considered in this article.
\end{abstract}

\maketitle

\section{Introduction}
The study of pointwise ergodic theorems in the noncommutative setting was initiated by Lance in his pioneering work \cite{La76}, where he proved a pointwise ergodic theorem for state-preserving automorphisms on von Neumann algebras.  Later, Yeadon \cite{Ye77} established bilaterally almost uniform convergence in noncommutative $L_1$-spaces by introducing weak type $(1,1)$ inequality
for the Ces\`aro averages
\begin{align*}
M_n x=\frac{1}{n} \sum_{k=0}^n T^k x, \qquad x\in L_1(\M),   
\end{align*}
where $T$ is a Dunford-Schwartz operator $(DS$ operator in short$)$ on a semifinite von Neumann algebra $\M$.
This was the first weak-type maximal inequality established in the noncommutative setting. 
Much later, Junge and Xu \cite{JX07}  obtained the individual (pointwise) ergodic theorem by extending Yeadon's results \cite{Ye77} to almost all other noncommutative $L_p$-spaces through their celebrated theory of noncommutative maximal functions. 
In particular, their work established strong-type $(p,p)$ maximal inequalities in the noncommutative setting for $p>1$.  Several further developments have been made in this direction. We refer interested readers to \cite{HLW21, HLX24, HS18, HRW23} and the references cited therein.

On the other hand, weighted ergodic theorems provide an important extension of the classical unweighted theory. 
In the commutative setting, Yoshimoto \cite{Yo20} studied the modulated $(C,\alpha)$-weighted ergodic averages
\begin{align}\label{Eq:ModulatedStandardAverages}
M_n(w,T,\alpha)x = \frac{1}{A_n^\alpha}\sum_{k=0}^n A_{n-k}^{\alpha-1}\,w_k\,T^k x,    
\end{align}
where $T$ is a $DS$ operator, $\mathbf{w}=\{w_k\}_{k\geq 0}$ belongs to the weighted class $W_q^\alpha$, and $A_n^\alpha$ denotes the $(C,\alpha)$-coefficients.
Recently, Hong, Jing, and Zhang \cite{HJZ26} extended this study to the noncommutative setting, establishing mean ergodic convergence, strong- and weak-type maximal inequalities, and bilateral almost uniform (resp. almost uniform) convergence for Besicovitch classes on noncommutative $L^p$-spaces.
Their results cover $\alpha>0$ and $1\leq p\leq \infty$, with particular significance for $0<\alpha<1$, where the admissible range of $p$ depends on both $\alpha$ and the weight exponent $q$, and extend the earlier results of Yeadon \cite{Ye77, Ye80} and Junge-Xu \cite{JX07}.

Throughout this paper, $\M\subseteq \mathbf{B}(\mathcal{H})$ denotes a semifinite von Neumann algebra with separable predual and $\tau$ denotes a faithful normal semifinite $(f.n.s.$ in short$)$ trace on $\M$. Let $\Phi:[0,\infty)\rightarrow [0,\infty)$ be an Orlicz function such that $\Phi$ is continuous at $0$ with $\Phi(0)=0$ and $\Phi(t)>0$ for all $t>0$. Further, for $p\geq 1$,  $\Phi$ is said to be \emph{$p$-convex} if the function $ t\mapsto\Phi(t^\frac{1}{p})$,  $t\geq 0$, is convex. For $1\leq p\leq \infty$ and an Orlicz function $\Phi$,  let $L_p(\M)$ and $L_\Phi(\M)$ respectively denote the noncommutative $L_p$-space and noncommutative Orlicz space associated with the pair $(\M, \tau)$.

In a different extension replacing scalar weights with operator-valued weights, Bikram and Saha   \cite{BS24} considered (bounded) center-valued Besicovitch sequences and studied the corresponding weighted ergodic averages along subsequences of density one. They established individual ergodic theorems for these averages on noncommutative Orlicz spaces associated with $DS$ operators. More precisely, 
for a positive $DS$ operator $T$ and a sequence $\mathbf{k}=\{k_j\}_{j\geq 0}$ of natural numbers, they consider the following averages
\begin{align}\label{Eq:BikramSahaAverages1}
 A_n(\mathbf{b}, \mathbf{d}, x):= \frac{1}{n} \sum_{k=0}^{n-1} T^k (b_k x d_k),  \qquad	A_n(\mathbf{b}, x):= \frac{1}{n} \sum_{k=0}^{n-1} T^k (b_k x),
\end{align}
and 
\begin{align}\label{Eq:BikramSahaAverages2}
A_n^{\mathbf{k}}(\mathbf{b}, \mathbf{d}, x):= \frac{1}{n} \sum_{j=0}^{n-1}  T^{k_j} (b_{k_j} x d_{k_j}), \qquad 	A_n^{\mathbf{k}}(\mathbf{b}, x):= \frac{1}{n} \sum_{j=0}^{n-1} T^{k_j} (b_{k_j} x),
\end{align}
for all $n \in \mathbb{N}$ and $x \in L^1(\M) + \M$, 
where $\mathbf{b}=\{b_k\}_{k\geq 0}$ and $\mathbf{d}=\{d_k\}_{k\geq 0}$ are $\mathcal{Z}(\M)$-valued bounded
sequences. They prove a weak-type $(1,1)$ maximal inequality for the above averages. 
Using this inequality, they establish bilateral uniform equicontinuity in measure at zero on noncommutative Orlicz spaces. Then they apply a Banach principle to obtain individual convergence in noncommutative Orlicz spaces.
This result extends the works of O'Brien \cite{Ob21} 
on scalar-valued weighted ergodic averages and subsequential weighted ergodic averages which were considered on noncommutative $L_p$-spaces, $1\leq p<\infty$. Moreover, the results of \cite{BS24} extend the work of \c C\"omez and Litvinov \cite{CL13} on vector-valued weighted ergodic averages for $1<p<\infty$ to the setting of noncommutative Orlicz spaces.
For further developments in this direction, we refer the reader to 
\cite{LM01, CL06, CL13, CL15, Ob21, Ob23, Ob24, Ob26, Li24, Li26} and the cited references therein.

The purpose of this paper is to combine the two approaches developed by Hong, Jing, and Zhang \cite{HJZ26} and by Bikram and Saha \cite{BS24} into a unified framework. To proceed in this direction, we consider operator-valued weighted ergodic averages and their subsequential versions. Our framework captures the interplay between $(C,\alpha)$-type modulation and operator-valued weights. 
We study mean ergodic theorems, maximal inequalities, and pointwise convergence in $p$-convex noncommutative Orlicz spaces. We also investigate the corresponding convergence along subsequences of density one. The $p$-convexity of the underlying Orlicz function plays a crucial role in our approach.
The main novelty of our approach lies in allowing operator-valued weights in the study of modulated $(C,\alpha)$-ergodic averages. This provides a common framework that connects the scalar-weighted $(C,\alpha)$-ergodic theory of Hong, Jing, and Zhang \cite{HJZ26} with the operator-valued weighted and subsequential ergodic theory due to Bikram and Saha \cite{BS24}. In particular, our results extend both settings under suitable assumptions on the weights and the underlying Orlicz spaces. In a different direction, we extend the class of noncommutative Orlicz spaces for which the noncommutative Rota theorem holds, with $p$-convexity playing a crucial role in this extension. We refer to \cite{An06, Hu08, Hu09, Bi24} for related results in this direction.

To describe our setting, let $T$ be a $DS$ operator on $\M$.
For $\alpha>0$ and $1\leq q< \infty$, we define the von Neumann algebra valued weighted class $W_{q}^\alpha(\M)$ as collection of all sequences of operators $\mathbf{b}=\{b_k\}_{k\geq 0}\subseteq \M$ with the seminorm
\begin{align*}
\norm{b}_{W_{q}^\alpha(\M)}   =\left(\limsup_{n\rightarrow\infty} \frac{1}{A_n^\alpha}\sum_{k=0}^n\, A_{n-k}^{\alpha-1} \, \norm{b_k}_\infty^q \right)^\frac{1}{q}<\infty,
\end{align*}
where $(C,\alpha)$-coefficients $A_n^\alpha$ are given by
\begin{align*}
A_0^\alpha=A_n^0=1 \quad\text{ and }\quad A_n^\alpha=\dfrac{(\alpha+1)\cdots (\alpha+n)}{n!}, \quad\quad \alpha>-1, \quad n\geq 0.    
\end{align*}
When $q=\infty$, we use the convention that $W_{\infty}^\alpha(\M)=\ell_\infty(\M)$ equipped with the standard norm.
We also introduce the class of operator-valued Besicovitch sequences $\mathcal{U}_f(q,\alpha)\subseteq W_q^\alpha(\M)$, defined as a closure of trigonometric polynomials $($see, Section \ref{Subsec:OpValuedBesicovitchWeight}$)$ in $W_q^\alpha(\M)$, where $1\leq q<\infty$.

Given a complex sequence $\mathbf{w}=\{w_k\}_{k\geq 0}$, two sequences of operators $\mathbf{b}=\{b_k\}_{k\geq 0}\subseteq \M$ and $\mathbf{d}=\{d_k\}_{k\geq 0}\subseteq \M$, and a sequence $\mathbf{k}=\{k_j\}_{j\geq 0}$ of natural numbers, we consider the following modulated operator-valued weighted averages:
\begin{align}\label{Eq:ModulatedWeightedAverage}
&M_n(\mathbf{b}, \mathbf{d}, \mathbf{w}, T, \alpha) x
=\frac{1}{A_n^\alpha}\sum_{k=0}^n\, A_{n-k}^{\alpha-1} \,  w_k T^k(b_k x d_k),  \quad x\in L_1(\M)+\M, \\
\nonumber\noindent \text{and}\qquad\qquad\qquad&\\
&\label{Eq:ModulatedWeightedAverage2}
M_n^{\mathbf{k}}(\mathbf{b}, \mathbf{d}, \mathbf{w}, T, \alpha) x
=\frac{1}{A_n^\alpha}\sum_{j=0}^{n}\, A_{n-j}^{\alpha-1} \,  w_{k_j}T^{k_j}(b_{k_j} x d_{k_j}),  \quad x\in L_1(\M)+\M. 
\end{align}

\subsection{Main Results}
We now state the main results of this paper, including the mean ergodic theorem and maximal inequalities. Moreover, we provide almost uniform and bilateral almost uniform convergences of the weighted ergodic averages on suitable noncommutative Orlicz spaces. Finally, we state an extension of the class of noncommutative Orlicz spaces for which Rota theorem holds.

We begin with our first result on the mean ergodic theorem. 
\begin{thm}\label{Thm:ConvergenceInLpNormHongType}
Let $T$ be a $DS$ operator on $\M$. For $\alpha > 0$, let $\mathbf{b}=\{b_k\}_{k \geq 0}$ and $\mathbf{d}=\{d_k\}_{k \geq 0}$ belong to $\mathcal{U}_f(1,\alpha)$ such that at least one of them is bounded. Then for all $x \in L_p(\mathcal{M})$ with $1<p<\infty$, the sequence $\{M_n(\mathbf{b},\mathbf{d},\mathbf{1},T,\alpha)x\}_{n\geq 0}$ converges in $L_p(\mathcal{M})$-norm. Moreover, if $\M$ is finite, the same result holds for $p=1$. In the case $p=\infty$, when $\M$ is finite, the above convergence is with respect to the $w^*$-topology.
\end{thm}

One of the main purposes of this paper is to study pointwise convergence in Orlicz spaces. For this, we first establish a maximal inequality. Our result is analogous to the maximal inequality obtained in \cite{HJZ26}, but its proof requires some new techniques that are not available in \cite{HJZ26}.

\begin{thm}\label{Thm:MaximalInequalityForVectorWeightedSequenceSpace}
Let $\alpha >0, 1 \leq q \leq \infty$ and $q^\prime$ be the conjugate of $q$, that is, $\frac{1}{q}+\frac{1}{q^\prime}=1$. Let $T$ be a positive $DS$ operator on $\M$, and
$\mathbf{b}=\{b_k\}_{k\geq 0}$ and $\mathbf{d}=\{d_k\}_{k\geq 0}$ belong to $W_q^{\alpha}(\mathcal{Z}(\M))$ such that at least one of them is bounded.
Then the following statements hold:
\begin{enumerate}[label=(\roman*)]
\item If $1<q\leq \infty$, then for every $x \in L_p(\mathcal{M})$ with $max\{\frac{q^\prime}{\alpha}, q^\prime\} < p \leq \infty$, there exists a positive constant   $C_{\alpha,p, q}$ depending only on $\alpha, p$ and $q$, such that
\begin{align*}
\norm{\left(M_n(\mathbf{b}, \mathbf{d}, \mathbf{1},T, \alpha)x\right)_{n\geq 0}}_{L_p(\mathcal{M}; \ell_{\infty})} \leq C_{\alpha,p,q}\norm{x}_p.
\end{align*}
Moreover, for $x\in L_p(\mathcal{M})$ with $\max\{\frac{2q^\prime}{\alpha},2q^\prime\}<p\leq\infty$, we have
\begin{align*}
\norm{\left(M_n(\mathbf{b}, \mathbf{d}, \mathbf{1},T, \alpha)x\right)_{n\geq 0}}_{L_p(\mathcal{M}; \ell_{\infty}^c)} \leq \sqrt{C_{\alpha,p,q}}\norm{x}_p.   
\end{align*}

\item Let $\alpha \geq 1$. If $1<q\leq \infty$, then for every $x \in L_p(\mathcal{M})$ with $p = q^\prime$ $($resp. $p=2q^\prime)$ and for 
$\varepsilon>0$, there exists a projection $e \in \mathcal{P}(\M)$ such that
\begin{align*}
 \tau(e^\perp)\leq \frac{16\norm{x}_p^p}{\varepsilon^p} \quad \text{ and }\quad &\sup_{n\geq 0}\norm{e \left(M_n(\mathbf{b},\mathbf{d}, \mathbf{1}, T, \alpha) x\right) e}_\infty
 \leq \varepsilon \\
 (\text{resp.  }&\sup_{n\geq 0}\norm{ \left(M_n(\mathbf{b},\mathbf{d}, \mathbf{1}, T, \alpha) x\right) e}_\infty\lesssim_\alpha \varepsilon).
\end{align*}

\item If $q=1$ and $\alpha\geq 1$, then for every $x \in L_p(\mathcal{M})$ with $p = q^\prime$ $($resp. $p=2q^\prime)$,
the above weak type inequality is replaced by the strong type in $(i)$.
\end{enumerate}
\end{thm}

Analogous to Theorem \ref{Thm:ConvergenceInLpNormHongType}, we now state the corresponding almost uniform convergence result, which will play a crucial role in proving the pointwise ergodic theorem in noncommutative Orlicz spaces. Notably, this result does not require the weights to be central.

\begin{thm}\label{Thm:ConvergentForDenseSet}
Let $\alpha>0$ and let $T$ be a positive $DS$ operator on $\M$. Let $\mathbf{b}=\{b_k\}_{k\geq 0}$ and $\mathbf{d}=\{d_k\}_{k\geq 0}$ be elements of 
$\mathcal{U}_f(1,\alpha)$ such that at least one of them is bounded. Then the following statements hold:
\begin{enumerate}[label=(\roman*)]
    \item For every $x \in L_1(\M)\cap\M$, the sequence $\{M_n(\mathbf{b}, \mathbf{d},\mathbf{1},T, \alpha)x\}_{n\geq 0}$ converges almost uniformly.
    \item If $\alpha\geq 1$ and $\mathbf{k} = \{k_j\}_{j\geq 0}$ is an increasing sequence of natural numbers with density $1$, then the sequence $\{M_n^{\mathbf{k}}(\mathbf{b}, \mathbf{d},\mathbf{1},T, \alpha)x\}_{n\geq 0}$ converges almost uniformly for all $x \in L_1(\M)\cap\M$.
\end{enumerate}
\end{thm}

The maximal inequality has two important consequences. First, it implies that the corresponding family of ergodic averages is bilaterally uniformly equicontinuous in measure at zero on $L_p(\M)$. Second, when the Orlicz function $\Phi$ satisfies the $p$-convexity condition, with the help of maximal inequality and a suitable spectral decomposition, we obtain bilateral uniform equicontinuity in measure at zero on $L_\Phi(\M)$ for the averages considered in Eqs. \eqref{Eq:ModulatedWeightedAverage} and \eqref{Eq:ModulatedWeightedAverage2}.
The noncommutative Banach principle of Litvinov \cite{Li12} and Theorem \ref{Thm:ConvergentForDenseSet} allow us to prove the following pointwise ergodic theorem in noncommutative Orlicz spaces.

\begin{thm}\label{Thm:AUconvergenceInOrliczSpace}
For $\alpha \geq 1$ and $1\leq q\leq\infty$, let $q^\prime$ denote the conjugate index of $q$, that is $\frac{1}{q}+\frac{1}{q^\prime}=1$. Let $T$ be a positive $DS$ operator on $\M$. Suppose $\mathbf{b}=\{b_k\}_{k\geq 0}$ and $\mathbf{d}=\{d_k\}_{k\geq 0}$ are sequences in $\mathcal{U}_f(1,\alpha)\cap W_{q}^\alpha(\mathcal{Z}(\M))$  such that at least one of them is bounded and $\mathbf{k} = \{k_j\}_{j\geq 0}$ is an increasing sequence of natural numbers with density $1$. If $\Phi$ is a $p$-convex Orlicz function satisfying $\Delta_2$-condition with $\max \{\frac{q^\prime}{\alpha}, q^\prime\}<p<\infty$, then for every $x\in L_\Phi(\M)$, the sequence
$\{M_n^\mathbf{k}(\mathbf{b}, \mathbf{d}, \mathbf{1}, T, \alpha) x\}_{n\geq 0}$ converges bilateral almost uniformly in $L_\Phi(\M)$.
\end{thm}
We note that the restriction 
$\alpha\geq 1$ in Theorem \ref{Thm:AUconvergenceInOrliczSpace} is not merely technical: our method for the subsequential case genuinely requires it (see Remark \ref{Rem:alphaLessThan1}), and the case $0<\alpha<1$ remains open.

Finally, as an application of our $p$-convexity result, we extend the noncommutative Rota theorem of Bikram \cite{Bi24} to a broader class of noncommutative Orlicz spaces.
\begin{thm}[Noncommutative Rota theorem]\label{Thm:RotaTheoremInourSetting}
Let $\M$ be a finite von Neumann algebra equipped with a faithful normal tracial state $\tau$. Let
$T:(\M,\tau)\longrightarrow(\M,\tau)$ be a factorizable $\tau$-preserving Markov operator. Let $\Phi:[0,\infty)\rightarrow[0,\infty)$ be an Orlicz function such that the function
\begin{align*}
t\longmapsto \Phi(t^{1/p}),\qquad t\geq 0,    
\end{align*}
is convex for some $p>1$. Then, for every $x\in L_\Phi(\M)$, the sequence
$\big\{T^n(T^*)^n(x)\big\}_{n\geq 0}$ is bilateral almost uniformly convergent. Furthermore, if the function
$\widetilde{\Phi}(t)=\Phi(\sqrt{t})$, $t\geq 0$, 
such that
\begin{align*}
 t\mapsto \widetilde{\Phi}(t^{1/p})   
\end{align*}
is convex, then, for every $x\in L_\Phi(\M)$, the sequence $\big\{T^n(T^*)^n(x)\big\}_{n\geq 0}$ converges almost uniformly. 
\end{thm}

\subsection{Methodology and Novelty of Our Work}\label{SubSec:Method}
We now highlight the main ideas behind our approach.
\begin{itemize}
    \item One of the main novelties of this work is the replacement of scalar-valued weights with operator-valued weights in Eqs. \eqref{Eq:ModulatedWeightedAverage} and \eqref{Eq:ModulatedWeightedAverage2}. More precisely, motivated by the modulated $(C,\alpha)$-ergodic averages studied by Hong, Jing, and Zhang \cite{HJZ26}, we consider weights belonging to the operator-valued class $W_q^\alpha(\M)$, rather than the scalar-valued class $W_q^\alpha$. This leads to a substantially more general framework in which the elements involved in the weighted averages are allowed to belong to $\M$. In particular, our setting contains several previously studied weighted ergodic averages as special cases. 
    
    For e.g.,
    when $\alpha=1$ and $\mathbf{w}$ is the constant scalar sequence
    \begin{align*}
        \mathbf{w}=\{1\}_{k\geq 0}\subseteq \mathbb{C},    
    \end{align*}
    the averages in Eqs. \eqref{Eq:ModulatedWeightedAverage} and \eqref{Eq:ModulatedWeightedAverage2} reduce to those in Eqs. \eqref{Eq:BikramSahaAverages1}  and \eqref{Eq:BikramSahaAverages2} respectively, considered by Bikram and Saha \cite{BS24}.

    On the other hand, if
    \begin{align*}
        \mathbf{d}=\mathbf{1}=\{1\}_{k\geq 0}\subseteq\M\qquad\text{ and } \qquad \mathbf{b}=\{w_k \,1\}_{k\geq 0}\subseteq \M,   
    \end{align*}
    where $\mathbf{w}=\{w_k\}_{k\geq 0}\in W_q^\alpha(\mathbb{C})$ is a scalar sequence, then the average in Eq. \eqref{Eq:ModulatedWeightedAverage} 
    reduces precisely to the modulated $(C,\alpha)$-weighted average in Eq. \eqref{Eq:ModulatedStandardAverages} considered by Hong, Jing and Zhang \cite{HJZ26}. Thus, our framework simultaneously includes the weighted averages studied in \cite{BS24} and \cite{HJZ26} as special cases.

    \item Next, we relax the boundedness assumption imposed on the weights in Eqs. \eqref{Eq:BikramSahaAverages1}  and \eqref{Eq:BikramSahaAverages2}, respectively. 
    In \cite{BS24}, the underlying $\mathcal{Z}(\M)$-valued weight sequences are required to be bounded, whereas our framework allows one of the weight sequences to belong to the more general class $W_q^\alpha(\mathcal{Z}(\M))$.

    \item Moreover, a direct replacement of $T^k (x)$ by $T^k(b_k x d_k)$ in the proofs of \cite{HJZ26} may not be possible in general, since the family of maps $x \mapsto T^k(b_k x d_k)$ does not arise from iterates of a single $DS$ operator. To overcome this difficulty, we follow techniques similar to those developed in \cite{CL13, BS24}.  

    \item We establish the mean ergodic theorem, maximal inequalities, and bilateral almost uniform (resp. almost uniform) convergence for the new weighted ergodic averages on noncommutative Orlicz spaces. The maximal inequality plays a crucial role in establishing the corresponding pointwise ergodic theorem. In this direction, we extend the maximal inequality of \cite[Theorem 1.2]{HJZ26} to the operator-valued weighted setting in Theorem \ref{Thm:MaximalInequalityForVectorWeightedSequenceSpace}. A key step in its proof is to reduce estimates for $\M$-valued weighted averages to suitable estimates for scalar-weighted averages. This reduction allows us to apply the existing scalar-valued maximal inequalities.

    \item One of the main technical ingredients in our approach is the introduction of the $p$-convexity condition on the underlying Orlicz function into the study of pointwise ergodic theorems on noncommutative Orlicz spaces. The main tool in \cite{CL13} and \cite{BS24} is a weak-type $(1,1)$ inequality, whereas the results in our Theorem \ref{Thm:MaximalInequalityForVectorWeightedSequenceSpace} allow us with a weak-type $(p,p)$ inequality, where $p>1$. Therefore, in order to extend the techniques of \cite{CL13} and \cite{BS24} to our setting, the $p$-convexity of the underlying Orlicz function plays a crucial role. The $p$-convexity of the underlying Orlicz function yields a useful spectral decomposition: every positive element of a noncommutative Orlicz space with norm bounded by one can be dominated by the sum of an element with arbitrarily small operator norm and an element of $L_p(\M)$, whose $L_p$-norm is controlled by the Orlicz norm of the original element. With the help of maximal inequality and the above decomposition, we obtain bilateral uniform equicontinuity in measure at zero on $L_\Phi(\M)$. By applying the noncommutative Banach principle, we can extend convergence from a suitable dense subspace to pointwise convergence on the noncommutative Orlicz space.

    \item Another novelty lies in the treatment of subsequential convergence for $\alpha\geq 1$. In this case, the presence of the $(C,\alpha)$-coefficients $A_n^\alpha$ introduces an additional difficulty that is absent from the usual Ces\`aro averages. Our argument uses the asymptotic behavior of these coefficients to pass from the full sequence to subsequences of density one. This enables us to establish the corresponding convergence results along subsequences of density one for $\alpha\geq 1$. Whether the corresponding subsequential convergence holds for $0<\alpha<1$ remains open; see Remark \ref{Rem:alphaLessThan1}.

    \item In a different direction, we extend the underlying classes of noncommutative Orlicz spaces for which the noncommutative Rota theorem $($see \cite{Bi24} and \cite{An06}$)$ holds. Here again, $p$-convexity plays a crucial role in obtaining the required estimates.
\end{itemize}
Thus, our results provide both an operator-valued extension of modulated weighted ergodic theory and a broader Orlicz-space framework for noncommutative pointwise ergodic theorems.

\subsection{Outline}
The rest of the paper is organized as follows. In Section \ref{Sec:Preliminaries}, we discuss the necessary preliminaries that are required for this paper. In Section \ref{Sec:pConvexOrliczSpace}, we introduce the notions of $p$-convex Orlicz functions and establish a useful spectral decomposition (Proposition \ref{Prop:ToMakeSmallElementInOrliczSpace}) that will be used in Section \ref{Sec:MainResults} to deal with pointwise convergence in noncommutative Orlicz spaces. In Theorem \ref{Thm:RotaTheoremInourSetting}, we extend the class of noncommutative Orlicz spaces for which Rota theorem holds.
In Section \ref{Sec:TechLemma}, we introduce a class of $\M$-valued weighted sequences, which allows us to consider weighted averages that differ from those studied by Hong et al. \cite{HJZ26}. We prove the completeness of the aforesaid space in Theorem \ref{Thm:WqAlphaComplete}. We also introduce weighted Besicovitch sequences in this section.

Sections \ref{Sec:MaximalInequalityWeightedErgodicAverages} and \ref{Sec:MainResults} contain the main results of the paper. In Section \ref{Sec:MaximalInequalityWeightedErgodicAverages}, using the results of \cite{HJZ26}, we first establish a mean ergodic theorem (Theorem \ref{Thm:ConvergenceInLpNormHongType}). We then prove a maximal inequality (Theorem \ref{Thm:MaximalInequalityForVectorWeightedSequenceSpace}) for the corresponding weighted ergodic averages. In Section \ref{Sec:MainResults}, we establish pointwise convergence for a dense subset of noncommutative Orlicz space $($Theorem \ref{Thm:ConvergentForDenseSet}$)$. Finally, in Theorem \ref{Thm:AUconvergenceInOrliczSpace}, we further obtain pointwise convergence when the weights are taken from the center of the von Neumann algebra for $\alpha\geq 1$; the case $0<\alpha<1$ is discussed as an open problem in Remark \ref{Rem:alphaLessThan1}.

\section{Preliminaries}\label{Sec:Preliminaries}
In this section, we collect some preliminaries on noncommutative $L_p$-spaces, noncommutative Orlicz spaces, and $DS$ operators that are required for this paper. Throughout this paper, $\M\subseteq \mathbf{B}(\mathcal{H})$ denotes a semifinite von Neumann algebra with separable predual and $\tau$ denotes a faithful normal semifinite $(f.n.s.$ in short$)$ trace on $\M$. Let $\M_+$ and $\M_{s.a.}$ respectively denote the positive cone and self-adjoint part of the von Neumann algebra $\M$. The operator norm on $\M$ is denoted by $\norm{\cdot}_\infty$. We also denote by $\mathcal{P}(\M)$ the collection of all projections in $\M$. For detailed theory on von Neumann algebras, we refer the reader to \cite{SZ19, Ta02, Ta03}. Throughout this paper, we write $A\lesssim B$ if there exists an absolute constant $c$, independent of parameters, such that $A\leq cB$. Similarly, we write $A\lesssim_{\alpha} B$ if there exists a positive constant $c_{\alpha}$, depending only on $\alpha$, such that
\begin{align*}
A\leq c_{\alpha}B.
\end{align*}

\subsection{Noncommutative $L_p$-spaces}\label{SubSec:NCOrlicz}
In this section, we recall some basic facts
on noncommutative $L_p$-spaces associated 
with the pair $(\M,\tau)$. More details on
noncommutative $L_p$-spaces can be found 
in \cite{Hi21, FK86, PX03}.

A closed, densely defined operator $a$ affiliated with $\M$ $($denoted by $a\eta \M)$ is said to be $\tau$-measurable if, for any $\delta>0$, there exists a projection $e\in\mathcal{P}(\M)$
such that $\tau(e^\perp)\leq \delta$ and $\norm{ae}_\infty<\infty$. Let $L_0(\M)$ denote the collection of all $\tau$-measurable operators associated with the pair $(\M,\tau)$. 
Then $\tau$ extends naturally from $\mathcal{M}_+$, the positive cone of $\mathcal{M}$, to $L_0(\mathcal{M})_+$, the positive cone of $L_0(\mathcal{M})$, by setting
\begin{align*}
\tau(x)=\int_{0}^{\infty} \lambda\,\, d\tau(e_\lambda), \quad \quad x\in L_0(\M)_+,
\end{align*}
where $x=\int_{0}^{\infty} \lambda \,\,d e_\lambda$ denotes the spectral decomposition of $x$. Let $x,y \in L_0(\M)$. Then the strong sum and strong product of $x$ and $y$ are respectively denoted by $\overline{x+y}$ and $\overline{xy}$. It is well known that $L_0(\M)$ is a $\ast$-algebra with respect to strong sum, strong product, and adjoint. Throughout this article, we use the convention that $x+y$ and $xy$ will denote the strong sum and strong product, respectively.

For an operator $x\in L_0(\M)$, 
let $\abs{x}$ be the absolute value of the operator $x$.
For $0< p<\infty$, the noncommutative $L_p$-space $L_p(\M)$ associated with the pair $(\M,\tau)$ is defined by
\begin{align*}
L_p(\M)=\left\{x\in L_0(\M): \norm{x}_p:=\tau(\abs{x}^p)^{1/p}<\infty\right\}.
\end{align*}
For $p=\infty$, we set $L_\infty(\M):=\M$ equipped with the operator norm $\norm{\cdot}_\infty$ on $\M$. For $0<p<1$, the noncommutative $L_p$-space $L_p(\M)$ equipped with the quasi-norm $\norm{\cdot}_p$ becomes a quasi-Banach space, and becomes a Banach space for any $1\leq p\leq\infty$. Moreover, $\M\cap L_1(\M)$ is dense in $ L_p(\M)$ with respect to  $\norm{\cdot}_{p}$ for $1\leq p<\infty$.

\noindent For  $x\in L_0(\M)$ and $t>0$,  the $(t$-th$)$ generalized $s$-number $\mu_t(x)$ is defined by 
\begin{align*}
\mu_t(x)=\inf\left\{s>0:\tau(e_{(s,\infty)}(\abs{x}))\leq t\right\},
\end{align*}
where $e_{(s,\infty)}(\abs{x})$ is the spectral projection of $\abs{x}$ corresponding to the interval $(s,\infty)$. It is known $($see, e.g., \cite[Proposition 4.20]{Hi21}$)$ that, for $x\in L_0(\M)_+$ and any continuous non-decreasing function $f$ on $[0,\infty)$ satisfying $f(0)\geq 0$, one has
\begin{align*}
\tau(f(x))=\int_{0}^{\infty}\,f\left(\mu_t(x)\right) \, dt.
\end{align*}
The following remark will be useful later.
\begin{rem}\label{Rem:Commutivityofoperators}
Let $x \in L_0(\M)$ be a positive self-adjoint operator and let $b,d \in Z(\M)_+$. 
Then, where products are understood in the sense of strong products, $bxd$ is a positive self-adjoint operator. Moreover, 
\begin{align*}
 bxd = dbx=(db)^{1/2}x(db)^{1/2}\qquad\text{ and }\qquad bxd\leq\norm{b}_\infty\norm{d}_\infty x. 
\end{align*}
Indeed, the equality $bxd=dbx$ follows from the fact that $b,d$ are central, while the representation $(db)^{1/2}x(db)^{1/2}$ makes the positivity clear $($see, e.g., \cite[E9.21, pp. 268]{SZ19} and \cite[Proposition 2.2.12]{DPS23}$)$.
\end{rem}

\subsection{Noncommutative Orlicz spaces}
\noindent Let $\Phi: [0,\infty)\longrightarrow [0,\infty)$ be an Orlicz function $($also called an Young function$)$. Then
the noncommutative Orlicz space $L_\Phi(\M)$ is defined as the collection of all
$x\in L_0(\M)$ for which there exists $\alpha>0$ such that 
\begin{align*}
\tau(\Phi(\alpha\abs{x}))<\infty.    
\end{align*}
Then $L_\Phi(\M)$ equipped with the Luxemburg-Nakano norm:
\begin{align*}
\norm{x}_\Phi=\inf\left\{\alpha>0: \tau\left(\Phi\left(\frac{\abs{x}}{\alpha}\right)\right)\leq 1\right\}, \quad \quad x\in L_\Phi(\M),    
\end{align*}
becomes a Banach space. Throughout this paper, we consider Orlicz functions that are convex, continuous at $0$, satisfy $\Phi(0)=0$ and $\Phi(t)>0$ for all $t>0$. We denote by $L_\Phi(\M)_+$ the positive cone of $L_\Phi(\M)$, defined by
\begin{align*}
L_\Phi(\M)_+=\{x\in L_\Phi(\M): x\geq 0\}.    
\end{align*}
It is elementary to check that $\Phi$ is continuous and strictly increasing on $[0,\infty)$.
Note that for $\Phi(t)=t^p$ with $1\leq p<\infty$, the noncommutative Orlicz space $L_\Phi(\M)$ coincides isometrically with $L_p(\M)$.

\begin{defn}[$\Delta_2$-condition]
An Orlicz function $\Phi:[0,\infty)\rightarrow[0,\infty)$ is said to satisfy the $\Delta_2$-condition, if there exists a positive constant $k>0$ such that
\begin{align*}
 \Phi(2t)\leq k\Phi(t), \qquad t\geq 0. 
\end{align*}
For e.g., $\Phi(t)=t\log (1+t)$, $t\geq 0$, satisfy $\Delta_2$ condition with $k=4$. Indeed, since $1+2t\leq (1+t)^2$ for every $t\geq 0$, 
\begin{align*}
 \Phi(2t)= 2t\log(1+ 2t)\leq 4t\log (1+t)=4\Phi(t).  
\end{align*}
\end{defn}
The following lemma will be used in Theorem \ref{Thm:AUconvergenceInOrliczSpace}.
\begin{lem}\label{Lem:MIntersecL1DenseInOrliczspace}\cite[Proposition 2.3]{CL17}
If an Orlicz function $\Phi$ satisfy $\Delta_2$-condition, then $L_1(\M)\cap \M$ is dense in $L_\Phi(\M)$ in the norm $\norm{\cdot}_\Phi$.   
\end{lem}

It is known that $(L_\Phi(\M),\norm{\cdot}_\Phi)$ is a fully symmetric space with the Fatou property. Moreover, $(L_\Phi(\M),\norm{\cdot}_\Phi)$
is an exact interpolation space for the Banach couple $(L_1(\M), \M)$ $($see, \cite[Corollary 2.2]{CL17} for details$)$.
For details on noncommutative Orlicz spaces, we refer the readers to \cite{GL20}.

\subsection{Vector-valued noncommutative $L_p$-spaces}
For $1\leq p<\infty$, let $L_p(\M; \ell_\infty)$ be the space of all sequences $x=(x_n)_{n\geq 0}\subseteq L_p(\M)$ which admit the following factorization: there exist $a,b\in L_{2p}(\M)$ and a bounded
sequence $y=(y_n)_{n\geq 0}\subseteq \M$ such that $x_n=ay_n b$ for $n\geq 0$. Note that this factorization is not unique. Then $L_p(\M; \ell_\infty)$ equipped with the norm:
\begin{align*}
\norm{x}_{L_p(\M;\ell_\infty)}=\inf\left\{\norm{a}_{2p}\sup_{n\geq 0}\norm{y_n}_{\infty} \norm{b}_{2p}: x_n=ay_n b,\,\, a,b\in L_{2p}(\M), y_n\in \M,\,\,\forall n\geq 0\right\},
\end{align*}
becomes a Banach space, where the infimum is taken over all factorizations of $x$ defined above. We also denote $\norm{{\underset{n\geq 0}{{\sup}^{+}}} \, x_n}_p=\norm{x}_{L_p(\M;\ell_\infty)}$, and these two notations will be used interchangeably. Let $x=(x_n)_{n\geq 0}\subseteq L_p(\M)_+$ be a positive sequence. Then it can be shown that $x\in L_p(\M;\ell_\infty)$ if and only if there exists an element $a\in L_p(\M)_+$ such that $x_n\leq a$ for all $n$, and 
\begin{align*}
\norm{x}_{L_p(\M;\ell_\infty)}=\inf\left\{\norm{a}_{p}: \,\,\, x_n\leq a, \forall n\geq 0, \,\,a\in L_p(\M)_+\right\}.
\end{align*}
For $p=\infty$, we use the convention $L_\infty(\M;\ell_\infty):=\ell_\infty(\M)$ equipped with the supremum norm 
\begin{align*}
\norm{(x_n)_{n\geq0}}_{\ell_\infty(\M)}:=\sup_{n\geq 0}\norm{x_n}_\infty.    
\end{align*}

On the other hand, for $2\leq p\leq \infty$, let $L_p(\M;\ell_\infty^c)$ 
be the space of all sequences $x=(x_n)_{n\geq 0}\subseteq L_p(\M)$ which admit the following factorization: there exist an element $a\in L_{p}(\M)$ and a bounded sequence $y=(y_n)_{n\geq 0}\subseteq \M$ such that $x_n=y_n a$ for $n\geq 0$. We equip $L_p(\M; \ell_\infty^c)$ with the norm:
\begin{align*}
\norm{x}_{L_p(\M;\ell_\infty^c)}=\inf\left\{\norm{a}_{p}\sup_{n\geq 0}\norm{y_n}_{\infty}: x_n=y_n a,\,\, a\in L_{p}(\M), y_n\in \M,\,\,\forall n\geq 0\right\},
\end{align*}
where the infimum is taken over all factorizations of $x$ defined above. Then $L_p(\M; \ell_\infty^c)$, equipped with this norm, is a Banach space. It is known that $(x_n)_{n\geq 0}\in L_p(\M;\ell_\infty^c)$ if and only if $(x_n^* x_n)_{n\geq 0}\in L_{\frac{p}{2}}(\M;\ell_\infty)$. Moreover,
\begin{align*}
\norm{(x_n)_{n\geq 0}}_{L_p(\M;\ell_\infty^c)}=\norm{(x_n^* x_n)_{n\geq 0}}_{L_{\frac{p}{2}}(\M;\ell_\infty)}^\frac{1}{2}.
\end{align*}
For more details on vector-valued noncommutative $L_p$ spaces, we refer \cite{JX07, Pi98, Ju02, Mu03}.

\subsection{Dunford-Schwartz Operators}
In this section, we discuss the Dunford-Schwartz operators. We begin with the following definition. 
\begin{defn}
A linear map $T:L_{1}(\M)+\M\longrightarrow L_{1}(\M)+\M$ is called Dunford-Schwartz operator if it satisfies the following two conditions:
\begin{align*}
\norm{Tx}_\infty\leq \norm{x}_\infty, \quad\forall x\in \M\qquad\text{ and }\qquad  \norm{Tx}_1\leq \norm{x}_1,\quad\forall x\in L_1(\M).  
\end{align*}
In addition, if $T(x)\geq 0$ for all $x\geq 0$, then $T$ is called a positive Dunford-Schwartz operator. Throughout the paper, we use $DS$ $($resp. $DS^+)$ to denote a Dunford-Schwartz $($resp. a positive Dunford-Schwartz$)$ operator. 
\end{defn}
If $T$ is a $DS$ operator, then, since $(L_\Phi(\M),\norm{\cdot}_\Phi)$
is an exact interpolation space for the Banach couple $(L_1(\M), \M)$ $($see, \cite[Corollary 2.2]{CL17} for details$)$, we have 
\begin{align*}
  \norm{Tx}_\Phi\leq\norm{x}_\Phi, \qquad x\in L_\Phi(\M).  
\end{align*}
Consequently, $T(L_\Phi(\M))\subseteq L_\Phi(\M)$.

We now recall the following H\"older-type inequality for $DS^+$ operators $($see \cite[Lemma 4.1]{HJZ26}$)$. The main idea of its proof can be found in \cite[Lemma 3.1]{Ob24}.

\begin{lem}\label{Lem:HolderInequalityForOperator}
Let $T$ be a $DS^+$ operator on $\M$. Let $1<p,q<\infty$ be such that $\frac{1}{p}+\frac{1}{q}=1$. Then for every positive sequence $\{a_k\}_{k\geq 1}$ of real numbers and for any positive 
$x\in L_1(\M)+\M$, we have
\begin{align}\label{Eq:HolderInequality}
\frac{1}{n}\sum_{k=1}^{n} a_k T^k x\leq   \left(\frac{1}{n}\sum_{k=1}^{n} a_k^p\right)^\frac{1}{p}\left(\frac{1}{n}\sum_{k=1}^{n} T^{k}(x^q)\right)^\frac{1}{q}, \quad\forall n\geq 1.  
\end{align}
\end{lem}

\subsection{Density and Convergence}
In this section, we recall the definition of density for a sequence of natural numbers and discuss the types of convergence in noncommutative $L_p$-spaces. Throughout this article, we denote $\mathbb{N}_0:=\mathbb{N}\cup\{0\}$. 

\begin{defn}[Density]
Let $\mathbf{k} = \{k_j\}_{j \geq 0} \subseteq \mathbb{N}_0$. Then a non-negative real number $d$ is said to be the density $($resp. lower density$)$ of $\mathbf{k}$ if 
\begin{align*}
d=\lim_{n \rightarrow \infty}\frac{|\{0,1,2,\ldots,n\} \cap \mathbf{k}|}{n+1} 
\quad  \bigg(\text{resp. } \quad\liminf_{n \rightarrow \infty} \frac{|\{0,1,2,...,n\} \cap \mathbf{k}|}{n+1}\bigg).
\end{align*}
\end{defn}
\noindent It is known that if $\mathbf{k}$ has density $d>0$ then $ \lim_{n \to \infty} \frac{k_n}{n} = \frac{1}{d}$ and  $\mathbf{k}$ has lower density $d >0$ iff $ \sup_{n \in \mathbb{N}} \frac{k_n}{n}< \infty$ $($see, \cite[Lemma 40]{Ro94}$)$.

Now we recall the definition of almost uniform $(a.u.$ in short$)$ and bilateral almost   
uniform $(b.a.u.$ in short$)$.
\begin{defn}
Let $\{x_n\}$ be a sequence of operators in $L_0(\M)$. Then we say that $x_n$ converges $a.u.$ $($resp. $b.a.u.)$  to $x\in L_0(\M)$ if for any $\varepsilon>0$, there exists a projection $e\in\mathcal{P}(\M)$ such that 
\begin{align*}
\tau(e^\perp)\leq \varepsilon\qquad\text{ and }\qquad\lim_{n\rightarrow\infty}\norm{(x-x_n)e}_\infty=0 \quad (\text{resp. }\lim_{n\rightarrow\infty}\norm{e(x-x_n)e}_\infty=0).   
\end{align*}
\end{defn}

Let us state the following important lemma, which will be useful throughout the paper. Its proof can be found in \cite[Lemma 3]{LM01}.
\begin{lem}\label{Lem:DifferenceSmallAULemma}
If $\{x_n\}_{n \geq1}$ be a sequence in $L_0(\mathcal{M})$ such that for every $\varepsilon>0$, there exists a sequence $\{y_n\}_{n\geq 0} \subset L_0(\mathcal{M})$ which is $a.u.$ $($resp. $b.a.u.)$ convergent and satisfies
\begin{align*}
\norm{x_n - y_n}_{\infty} \leq \varepsilon, \quad \text{ for } n\geq N,    
\end{align*}
for some $N \in \mathbb{N}$. Then $\{x_n\}_{n \geq 0}$ is also $a.u.$ $($resp. $b.a.u.)$ convergent. 
\end{lem} 

The following definition will be used in Theorem \ref{Thm:buemAt0ForM_nbTx}.

\begin{defn}\cite[Definition 2.1]{Li12} $(u.e.m./b.u.e.m.)$
Let $X_0$ be a subset of a normed space $(X,\norm{\cdot})$ containing $0$. A family of additive maps $a_n: X\rightarrow L_0(\M)$, $n\in\mathbb{N}$, is said to be \textit{uniformly equicontinuous in measure} $(u.e.m.$ in short$)$ $($resp. \textit{bilaterally uniformly
equicontinuous in measure} $(b.u.e.m.$ in short$))$ at $0$ on $X_0$ if given any $\varepsilon>0$ and $\delta>0$, there is a $\gamma>0$ such that for every $x\in X_0$ with $\norm{x}<\gamma$, there exists a projection $e\in\mathcal{P}(\M)$ for which 
\begin{align*}
\tau(e^\perp)\leq\varepsilon\quad\text{ and }\quad \sup_{n}\norm{a_n(x)e}_\infty\leq \delta  \quad(\text{resp. }\sup_{n}\norm{e a_n(x)e}_\infty\leq \delta  ).  
\end{align*}
\end{defn}

\subsection{The $(C,\alpha)$-coefficients}
In this section, we discuss on $(C,\alpha)$-coefficients $A_n^\alpha$ for $\alpha>-1$. More details can be found in \cite{Yo20}. For $\alpha>-1$ and $n\in\mathbb{N}$, the $(C,\alpha)$-coefficients $A_n^{\alpha}$
have the following properties:
\begin{align}\label{Eq:CalphaCoeffient}
\nonumber&A_n^{\alpha} >0, \quad A_0^{\alpha} = A_n^0 = 1, \quad A_n^{\alpha-1}=A_{n}^\alpha-A_{n-1}^\alpha, 
\quad A_n^{\alpha} = \sum_{k=0}^n A_{n-k}^{\alpha-1}, \\
\nonumber&A_n^{\alpha}
=\binom{n+\alpha}{n} =\frac{(\alpha+1)\cdots(\alpha+n)}{n!}= \frac{\Gamma(n+\alpha+1)}{\Gamma(n+1)\Gamma(\alpha+1)}\\
& \lim_{n\rightarrow\infty}A_n^\alpha\frac{\Gamma(\alpha+1)}{n^\alpha}=1.
\end{align}
For $\alpha>0$, $A_n^{\alpha}$ is increasing with respect to $n$.
Moreover, 
\begin{align}\label{Eq:AnalphaInequality}
\frac{n^\alpha}{\Gamma(\alpha+1)} \leq A_n^\alpha \lesssim_{\alpha} (n+1)^\alpha.
\end{align}
On the other hand, if $-1<\alpha<0$, then $A_n^{\alpha}$ is decreasing with respect to $n$. In this case,
\begin{align*}
 \frac{(n+1)^{\alpha}}{\Gamma(\alpha+1)}\leq A_n^\alpha\leq \frac{n^\alpha}{\Gamma(\alpha+1)}.    
\end{align*}

\section{$p$-convex Orlicz functions}\label{Sec:pConvexOrliczSpace}
In this section, we discuss $p$-convex Orlicz functions. The same $p$-convexity condition will play a crucial role in both the pointwise ergodic theorem in Section \ref{Sec:MainResults} and the extension of Rota theorem presented below. We begin with the following definition.

\begin{defn}[$p$-convexity]\label{Defn:pconvex}
Let $1\leq p<\infty$. An Orlicz function $\Phi$ is said to be
\emph{$p$-convex Orlicz function} if the function
\begin{align*}
t\longmapsto\Phi\bigl(t^\frac{1}{p}\bigr),\qquad t\geq 0,
\end{align*}
is convex.
\end{defn}

Now we discuss some properties of $p$-convex Orlicz functions. To avoid any future confusion, we use the notation $\Phi^\frac{1}{p}(t):=\left[\Phi(t)\right]^\frac{1}{p}$, $t\geq 0$, for the rest of the paper. 
\begin{lem}\label{Lem:Ratio}
Let $p\geq 1$. If $\Phi$ is a $p$-convex Orlicz function, then the following hold.
\begin{enumerate}[label=(\roman*)]
\item The map $s\longmapsto\frac{\Phi(s)}{s^{p}}$
  is nondecreasing on $(0,\infty)$.
\item The map $s\mapsto \frac{\Phi^\frac{1}{p}(s)}{s}$ is nondecreasing on $(0,\infty)$.
\item We have $\Phi(\alpha s)\leq\alpha^{p}\,\Phi(s)$, for all $s\geq 0$ and $0\leq\alpha\leq 1$.
\end{enumerate}
\end{lem}
\begin{proof}
Since $\Phi(0)=0$, put $\Psi(t)=\Phi(t^\frac{1}{p})$, so that $\Psi$ is convex with $\Psi(0)=0$.  

\noindent $(i)$. We claim that $\frac{\Psi(t)}{t}$ is nondecreasing. If $0<t\leq u$, then write 
\begin{align*}
t=\frac{t}{u}\,u+\bigl(1-\frac{t}{u}\bigr)\cdot 0.    
\end{align*}
Since $\Psi(0)=0$ and $\Psi$ is convex, 
we have
\begin{align*}
\Psi(t)\leq\frac{t}{u}\Psi(u).    
\end{align*}
Substituting $t=s^{p}$ gives that 
$\frac{\Phi(s)}{s^{p}}$ is nondecreasing. This proves $(i)$ in the statement. 

\noindent $(ii)$. The statement in $(ii)$ follows by taking the $p$-th roots in $(i)$.

\noindent $(iii)$. If $\alpha=0$ or $s=0$, then there is nothing to prove. Hence, we may assume that $s> 0$ and $0<\alpha\leq 1$. Since
$\alpha s\leq s$, by $(i)$, we have
\begin{align*}
\frac{\Phi(\alpha s)}{(\alpha s)^{p}}\leq\frac{\Phi(s)}{s^{p}}\Longrightarrow \Phi(\alpha s)\leq\alpha^{p}\,\Phi(s).    
\end{align*}
This completes the proof.
\end{proof}

The following lemma, which will be useful for our purposes, is similar in spirit to \cite[Lemma 2.1]{CL17}. However, its proof differs from \cite[Lemma 2.1]{CL17}.

\begin{lem}\label{Lem:pConvex}
Let $\Phi$ be a $p$-convex Orlicz function. Then for any $\delta>0$, there exists $t>0$ such that 
\begin{align*}
t\Phi^\frac{1}{p}(u)\geq u, \quad\text{ whenever }u\geq \delta.
\end{align*}
In particular, we have $\lim_{u\rightarrow\infty} \Phi(u)=\infty$. 
\end{lem}
\begin{proof}
For any $\delta>0$, we have $\Phi(\delta)\neq 0$. Set
\begin{align*}
t:=\frac{\delta}{\Phi^\frac{1}{p}(\delta)}.
\end{align*}
By Lemma \ref{Lem:Ratio}, the map 
\begin{align*}
u\mapsto\frac{\Phi^\frac{1}{p}(u)}{u}    
\end{align*}
is nondecreasing. So for every $u\geq\delta$, we have
\begin{align*}
\frac{\Phi^{\frac{1}{p}}(\delta)}{\delta}\leq \frac{\Phi^{\frac{1}{p}}(u)}{u}, \quad \text{ that is}, \quad u\leq \frac{\delta}{\Phi^{\frac{1}{p}}(\delta)} \Phi^{\frac{1}{p}}(u).
\end{align*}
This shows that
\begin{align*}
u\leq t\Phi^\frac{1}{p}(u), \quad\text{ whenever }u\geq \delta.
\end{align*}
This completes the proof.
\end{proof}

The following proposition plays a key role in our analysis.
\begin{prop}\label{Prop:ToMakeSmallElementInOrliczSpace}
Let $\Phi$ be a $p$-convex Orlicz function. Let $x\in L_\Phi(\M)_+$
with $\norm{x}_\Phi\leq 1$. Then for any $\delta>0$, there exist $x_\delta\in \M$ with $\norm{x_\delta}_\infty\leq \delta$,  and $y\in L_p(\M)$ with $\norm{y}_p \leq \norm{x}_\Phi^\frac{1}{p}$ such that
\begin{align*}
x\leq x_\delta+ t y,   
\end{align*}
for some $t>0$.
\end{prop}
\begin{proof}
Let $x\in L_\Phi(\M)_+$ such that $\norm{x}_{\Phi}\leq 1$. By \cite[Lemma 5.40]{GL20}, we have $\Phi(x)\in L_1(\M)$ and
\begin{align*}
\norm{\Phi(x)}_1\leq \norm{x}_{\Phi}.   
\end{align*}
Let $x=\int_0^\infty \lambda\, de_\lambda$ be the spectral decomposition of $x$. Then, for $\delta>0$, let $t>0$ be as in Lemma \ref{Lem:pConvex}. Then, we have
\begin{align}\label{Eq:ToMakeSmallElementInOrliczSpace}
x
=\int_0^{\delta} \lambda\, de_\lambda \,+\, \int_{\delta}^{\infty} \lambda\, de_\lambda 
\underset{(\text{Lemma \ref{Lem:pConvex}})}{\leq} \int_0^{\delta} \lambda\, de_\lambda \,+\, t\int_{\delta}^{\infty} \Phi^\frac{1}{p}(\lambda)\, de_\lambda 
\leq x_\delta +ty,
\end{align}
where $x_\delta=\int_0^{\delta} \lambda\, de_\lambda $ and $y=\Phi^\frac{1}{p}(x)$. 

Note that $\norm{x_\delta}_\infty\leq\delta$ and $y\in L_p(\M)$ with $\norm{y}_p^p=\tau(\Phi(x))=\norm{\Phi(x)}_1\leq \norm{x}_\Phi$. This completes the proof.
\end{proof}

Proposition \ref{Prop:ToMakeSmallElementInOrliczSpace} is the key technical ingredient underlying both main applications of this section. It will be used directly in Section \ref{Sec:MainResults} to establish the pointwise ergodic theorem (Theorem \ref{Thm:AUconvergenceInOrliczSpace}) and again in the proof of Theorem \ref{Thm:RotaTheoremInourSetting} below to extend the noncommutative Rota theorem.

\begin{rem}\label{Rem:RotaTheorem}
\noindent $(i)$. We first observe that Lemma \ref{Lem:pConvex}
and Proposition \ref{Prop:ToMakeSmallElementInOrliczSpace} hold whenever the function $s\longmapsto\frac{\Phi(s)}{s^{p}}$
is non decreasing on $(0,\infty)$. However, as shown below in Example \ref{Exmp:2ConvexisSmallerClass}, the condition that the function
$s\longmapsto\frac{\Phi(s)}{s^p}$
is non-decreasing on $(0,\infty)$ does not, in general, imply that $\Phi$ is $p$-convex. Moreover, in \cite{JZZ23}, the authors have shown that if an Orlicz function $\Phi$ satisfies condition $(i)$ of Lemma \ref{Lem:Ratio}, then $\Phi$ is equivalent to a $p$-convex function.

\vspace{2mm}

\noindent $(ii)$. There is a different definition of a $p$-convex Orlicz function considered in \cite{Bi24}. More precisely, an Orlicz function $\Phi:[0,\infty)\rightarrow [0,\infty)$ is said to be $p$-convex in the sense of \cite{Bi24}, if the function $\Phi^\frac{1}{p}$, defined by
\begin{align*}
\Phi^\frac{1}{p}(t)=[\Phi(t)]^\frac{1}{p}, \qquad t\geq 0,    
\end{align*}
is convex.  Note that Lemma \ref{Lem:pConvex}
and Proposition \ref{Prop:ToMakeSmallElementInOrliczSpace} hold for the class of $p$-convex Orlicz functions   considered in \cite{Bi24}. Indeed, Lemma \ref{Lem:pConvex} can be proved using an argument similar to that in \cite[Lemma 2.1]{CL17}. The proof of Proposition \ref{Prop:ToMakeSmallElementInOrliczSpace}
remains unchanged. On the other hand, we show below in Proposition \ref{Prop:BikramConvexImplyOurConvex} that the above definition of $p$-convexity implies our definition of $p$-convexity. Moreover, the example below $($see, Example \ref{Exmp:Our2ConvexDoesNOtImplyBikram2Convex1}$)$ shows that the converse does not hold.
\end{rem}

\begin{prop}\label{Prop:BikramConvexImplyOurConvex}
Let $p\geq 1$ and let $\Phi$ be an Orlicz function. Define 
\begin{align*}
f(t):=\Phi^\frac{1}{p}(t)\qquad\text{ and }\qquad    g(t):=\Phi(t^\frac{1}{p}), 
\qquad t\geq 0.
\end{align*}
If $f$ is convex, then $g$ is convex.
\end{prop}

\begin{proof}
Let us assume that $f$ is convex. We show that $g$ is convex.
In the view of \cite[E. 24, pp. 101]{Ru76}, it is enough to show 
\begin{align*}
g\left(\frac{s+t}{2}\right)\leq \frac{g(s)+g(t)}{2}, \qquad s,t\geq 0.  
\end{align*}
Let $s^\frac{1}{p}=a$ and $t^\frac{1}{p}=b$, so that $s=a^p$ and $t=b^p$. Then the above inequality is equivalent to
\begin{align*}
g\left(\frac{a^p+b^p}{2}\right)\leq \frac{g(a^p)+g(b^p)}{2}.    
\end{align*}
Set 
\begin{align*}
    c=\left(\frac{a^p+b^p}{2}\right)^\frac{1}{p}, \qquad t_1=\frac{a^{p-1}}{2c^{p-1}} \qquad\text{ and }\qquad t_2=\frac{b^{p-1}}{2c^{p-1}}.
\end{align*}
Then, 
\begin{align*}
c=t_1 a+ t_2 b \qquad\text{ and }\qquad
t_1+t_2\underset{(\text{{H}\"older})}{\leq} 2^\frac{1}{p} (t_1^q+t_2^q)^\frac{1}{q}  \leq 1.   
\end{align*}
Since $f$ is convex, we have
\begin{align*}
f(c)
=f(t_1 a+ t_2 b)
\leq t_1 f(a)+t_2 f(b)
= & (2^\frac{1}{p}t_1) \frac{f(a)}{2^\frac{1}{p}} +(2^\frac{1}{p}t_2) \frac{f(b)}{2^\frac{1}{p}}\\ 
\underset{(\text{{H}\"older})}{\leq}& 
\left((2^\frac{1}{p}t_1)^q+(2^\frac{1}{p}t_2)^q\right)^\frac{1}{q} \left(
\frac{f(a)^p}{2} + \frac{f(b)^p}{2}\right)^\frac{1}{p}\\
{\leq}& 
\left(
\frac{f(a)^p}{2} + \frac{f(b)^p}{2}\right)^\frac{1}{p}\\
=&\left(
\frac{\Phi(a)+\Phi(b)}{2}\right)^\frac{1}{p}\\
=&\left(
\frac{g(a^p)+g(b^p)}{2}\right)^\frac{1}{p}.
\end{align*}
Consequently, the above equation becomes
\begin{align*}
g\left(\frac{a^p+b^p}{2}\right)
=[f(c)]^p
\leq
\frac{g(a^p)+g(b^p)}{2}.
\end{align*}
This completes the proof.
\end{proof}

\begin{exmp}\label{Exmp:Our2ConvexDoesNOtImplyBikram2Convex1}
Let $\Phi:[0,\infty)\rightarrow [0,\infty)$ be defined by
 \begin{align*}
\Phi(t)
=
\begin{cases}
t^2, \qquad &0\leq t\leq 1;\\ 
2t^2-1, \qquad &t\geq 1.
\end{cases}     
\end{align*}   
Then, $\Phi$ is convex, strictly increasing satisfying $\Phi(0)=0$. 

The Orlicz function $\Phi$ satisfies the $\Delta_2$-condition. Indeed, for $t>0$, we have
\begin{align*}
\frac{\Phi(2t)}{\Phi(t)}
=
\begin{cases}
4, & 0<t\leq \frac{1}{2},\\[2mm]
8-\dfrac{1}{t^2}, & \frac{1}{2}<t\leq 1,\\
4+\dfrac{3}{2t^2-1}, & t>1.
\end{cases}
\end{align*}
Note that the right-hand side of each expression is at most $7$. Thus, 
\begin{align*}
\Phi(2t)\leq 7\Phi(t), \qquad t\geq 0.
\end{align*}
Moreover, we have
\begin{align*}
\Phi(t^\frac{1}{2})
=
\begin{cases}
t, \qquad &0\leq t\leq 1;\\ 
2t-1, \qquad &t\geq 1.
\end{cases}
\end{align*}
Thus, $\Phi$ is $2$-convex in our sense, that is the map $t\longmapsto\Phi(t^\frac{1}{2})$ is convex.
However, the function
\begin{align*}
\Psi(t):=(\Phi(t))^\frac{1}{2}
=
\begin{cases}
t, \qquad &0\leq t\leq 1;\\ 
\sqrt{2t^2-1}, \qquad &t\geq 1,
\end{cases}     
\end{align*}
is not convex, since for $t>1$, we have
\begin{align*}
\Psi^{\prime\prime}(t)
=\frac{-2}{(2t^2-1)^{\frac{3}{2}}}<0. 
\end{align*}
This $\Phi$ is not a $2$-convex Orlicz function in the sense of \cite{Bi24}. 
\end{exmp}

\begin{exmp}\label{Exmp:2ConvexisSmallerClass}
Let $\Phi:[0,\infty)\rightarrow [0,\infty)$ be defined by
\begin{align*}
 \Phi(t)
 =\begin{cases}
     \frac{1}{2} t^2, \qquad&\text{ if } 0\leq t\leq 1;\\
     2t-\frac{3}{2}\qquad&\text{ if } 1\leq t\leq \frac{3}{2};\\
     t^2-\frac{3}{4}\qquad&\text{ if }  t\geq \frac{3}{2}.
 \end{cases}
\end{align*}
Then, it is easy to see that $\Phi$ is an Orlicz function
which satisfies $\Delta_2$-condition. Moreover, $\Phi$ satisfies the condition $(i)$ in Lemma \ref{Lem:Ratio} for $p=2$, but $\Phi$ is not $2$-convex.
\end{exmp}

\subsection{Noncommutative Rota Theorem}
We briefly recall the setting of the noncommutative version of Rota theorem, referring to \cite{An06, Bi24} for the relevant definitions and terminology. Let $\M$ be a finite von Neumann algebra equipped with a faithful normal tracial state $\tau$, and let $T:\M\rightarrow\M$ be a factorizable $\tau$-preserving Markov operator. Rota theorem in this setting concerns the pointwise convergence of the sequence of iterates of $T$ and its adjoint. Anantharaman--Delaroche \cite{An06} established such a result on noncommutative $L_p$-spaces for $p>1$.

In the Orlicz-space setting, Hu \cite{Hu09} established Rota theorem on $L\log^\alpha L$ for $\alpha\geq 2$, using a maximal inequality for noncommutative martingales from $L\log^\alpha L$ into $L_1$. Moreover, the exponent $\alpha=2$ is optimal for this maximal inequality \cite{Hu09}. Most recently, Bikram \cite{Bi24} extended Rota theorem to a class of noncommutative Orlicz spaces that includes $L\log L(\M)$. His approach is based on the $L_p$-space result of Anantharaman-Delaroche \cite{An06} and \cite{Bu02}, together with a result of Litvinov \cite{Li12}. 

We now recall the noncommutative Rota theorem established in \cite{Bi24}.

\begin{thm}[Noncommutative Rota theorem {\cite[Theorem 1.6]{Bi24}}]\label{Thm:RotaTheoremPanchuDa}
Let $\M$ be a von Neumann algebra equipped with a faithful normal tracial state $\tau$. Let
$T:(\M,\tau)\longrightarrow(\M,\tau)$ be a factorizable $\tau$-preserving Markov operator. Let $\Phi:[0,\infty)\rightarrow[0,\infty)$ be an Orlicz function such that the function
\begin{align*}
t\longmapsto [\Phi(t)]^{1/p},\qquad t\geq 0,    
\end{align*}
is convex for some $p>1$. Then, for every $x\in L_\Phi(\M,\tau)$, the sequence
$\big\{T^n(T^*)^n(x)\big\}_{n\geq 0}$ converges $b.a.u.$. Furthermore, if the function
$\widetilde{\Phi}(t)=\Phi(\sqrt{t})$, $t\geq 0$, 
such that
\begin{align*}
 t\mapsto [\widetilde{\Phi}(t)]^{1/p}   
\end{align*}
is convex, then, for every $x\in L_\Phi(\M)$, the sequence $\big\{T^n(T^*)^n(x)\big\}_{n\geq 0}$ converges $a.u.$
In particular, $\{T^n(T^*)^n(x)\}_{n\geq 0}$ converges $b.a.u.$ for every
$x\in L\log L(\M)$.
\end{thm}

Following the approach of \cite{Bi24}, we use our $p$-convexity results to extend the class of noncommutative Orlicz spaces for which Rota theorem holds. We now present the proof of Theorem \ref{Thm:RotaTheoremInourSetting}. The key ingredient is Proposition \ref{Prop:ToMakeSmallElementInOrliczSpace}
which provides the spectral decomposition estimate corresponding to \cite[Eq.~(7), Lemma 3.9]{Bi24} under our $p$-convexity assumption. This allows us to adapt the argument of \cite{Bi24} to the present setting.

\begin{proof}[\textbf{Proof of Theorem \ref{Thm:RotaTheoremInourSetting}}]
In \cite[Theorem 1.6; see also Theorem 3.10]{Bi24}, the author proves  
Rota theorem in noncommutative Orlicz spaces using \cite[Lemma 3.9]{Bi24}. 
Note that the proof of \cite[Theorem 1.6]{Bi24} relies 
on \cite[Eq.~(7), Lemma 3.9]{Bi24}. In our setting, the corresponding estimate
is provided by Proposition \ref{Prop:ToMakeSmallElementInOrliczSpace}.
More precisely, Proposition \ref{Prop:ToMakeSmallElementInOrliczSpace} 
can be used to establish \cite[Eq.~(7), Lemma 3.9]{Bi24} under our $p$-convexity
assumption. With this modification, the rest of the proof will 
follow as in \cite[Theorem 1.6; see also Theorem 3.10]{Bi24}.
\end{proof}

\begin{rem}
The noncommutative Bufetov ergodic theorem \cite[Theorem 1.4; see also Theorem 4.4]{Bi24} also remains valid under our $p$-convexity assumption. Indeed, the proof in \cite{Bi24} relies on Theorem
\ref{Thm:RotaTheoremPanchuDa}. Replacing this theorem with Theorem \ref{Thm:RotaTheoremInourSetting}, the same argument yields the corresponding result on the noncommutative Bufetov ergodic theorem for the $p$-convex Orlicz functions considered in this paper. It is worth mentioning that recently, the authors in \cite{GMR26}, observe the failure of noncommutative Rota theorem with $L\log L$ endpoint.
\end{rem}

\section{Weighted Classes and Besicovitch weight}\label{Sec:TechLemma}

\subsection{Weighted Classes}
In this section, we define $\M$-valued weighted classes for $(C,\alpha)$-ergodic averages and discuss their properties. 
We begin with the following definition. 

\begin{defn}[Weight Classes]\label{Defn:WelytypeSpace}
Let $\alpha>0$ and $1\leq q\leq \infty$. For $\mathbf{b}=\{b_k\}_{k\geq 0}\subseteq \M$, we 
will write
\begin{align*}
\norm{\mathbf{b}}_{W_q^\alpha(\M)}
=\begin{cases}
&\left(\limsup_{n\rightarrow\infty} \frac{1}{A_n^\alpha}\sum_{k=0}^n\, A_{n-k}^{\alpha-1} \, \norm{b_k}_\infty^q \right)^\frac{1}{q},\qquad\text{ if }1\leq q<\infty;\\
&\sup_{k\geq 0}\norm{b_k}_\infty, \qquad\qquad\qquad\qquad\qquad\qquad\text{ if }q=\infty.
\end{cases}
\end{align*}
Then we define $\M$-valued $W_{q}^\alpha$-space as
\begin{align*}
W_{q}^\alpha(\M)
=\left\{\mathbf{b}=\{b_k\}_{k\geq 0}\subseteq \M:\norm{\mathbf{b}}_{W_q^\alpha(\M)}
<\infty\right\}.   
\end{align*}
We write $\mathbf{1}=\{1\}_{k\geq 0}$ to denote the sequence consisting of identity operators. When $\M=\mathbb{C}$, we simply write $W_{q}^\alpha$ instead of $W_{q}^\alpha(\mathbb{C})$.
\end{defn}
\noindent It is easy to see that 
\begin{align*}
\mathbf{b}\in W_{q}^\alpha(\M) \quad \text{ if and only if }\quad    \mathbf{w}:=\{\norm{b_k}_{\infty}\}_{k\geq 0}\in W_q^\alpha. 
\end{align*}
On the other hand, if $\mathbf{w}=\{w_k\}_{k \geq 0} \in W_q^{\alpha}$, then, defining $b_k = w_k 1$, $k\geq 0$, we have $\mathbf{b}=\{b_k\}_{k \geq 0} \in W_q^{\alpha}(\mathcal{M})$. Therefore, 
$W_q^\alpha$ is naturally embedded in $W_q^\alpha(\M)$.

Note that $W_{\infty}^\alpha(\M)
=\ell_\infty(\M)$. By definition, it follows that $\norm{\cdot}_{W_\infty^\alpha(\M)}$ is a genuine norm on $W_{\infty}^\alpha(\M)$ and that $W_{\infty}^\alpha(\M)$ becomes a Banach space with respect to the norm $\norm{\cdot}_{W_\infty^\alpha(\M)}$.
For $1\leq q<\infty$, one can show that $W_{q}^\alpha(\M)$
is a vector space and that $\norm{\cdot}_{W_q^\alpha(\M)}$ is a semi-norm on $W_{q}^\alpha(\M)$. 
We show below that with respect to the topology generated by this semi-norm $\norm{\cdot}_{W_q^\alpha(\M)}$, the space $W_{q}^\alpha(\M)$ is complete. We start with the following elementary lemma.

\begin{lem}\label{Lem:bk=0implynorm0}
For $\alpha>0$, let $\mathbf{b}=\{b_k\}_{k\geq 0}\in W_q^\alpha(\M)$ such that $b_k=0$ for all $k\geq N$, for some $N$. Then $\norm{b}_{W_q^\alpha(\M)}=0$.  
\end{lem}
\begin{proof}
For $n\geq N$, we have
\begin{align*}
\norm{b}_{W_q^\alpha(\M)}^q
=\limsup_{n\rightarrow\infty} \sum_{k=0}^N \frac{A_{n-k}^{\alpha-1}}{A_n^\alpha} \norm{b_k}_\infty^q
\leq \left(\sup_{0\leq k\leq N} \norm{b_k}_\infty^q\right) \sum_{k=0}^N \limsup_{n\rightarrow\infty}\frac{A_{n-k}^{\alpha-1}}{A_n^\alpha}
=0.
\end{align*}
This completes the proof.
\end{proof}

\begin{lem}\label{Lem:TruncationWithApproximation} Let $\alpha>0$ and $1\leq q<\infty$. 
For $m\geq 1$, let $\mathbf{b}^{(m)}\in W_q^\alpha(\M)$ be such that $\norm{\mathbf{b}^{(m)}}_{ W_q^\alpha(\M)}\leq \epsilon_m$. Then there exist $N_1<N_2<\cdots <N_m<\cdots$ with $N_m\geq m$ such that the element $\mathbf{d}^{(m)}=\{{d}^{(m)}_k\}_{k\geq 0}$ defined by
\begin{align*}
{d}^{(m)}_k
=\begin{cases}
b^{(m)}_k, \qquad\text{ if } k\geq N_m;\\ 
0, \qquad\quad \text{ if } k<N_m,  
\end{cases} 
\end{align*}
is in $W_q^\alpha(\M)$ and satisfying
\begin{align}\label{Eq:ApproximationByTrunction}
\norm{\mathbf{d}^{(m)} -\mathbf{b}^{(m)}}_{ W_q^\alpha(\M)}= 0 
\qquad\text{ and }\qquad  
\sup_{n\geq 0} \left(\frac{1}{A_n^\alpha}\sum_{k=0}^n A_{n-k}^{\alpha-1} \norm{d^{(m)}_k}_{\infty}^q\right)^\frac{1}{q} \leq 2\epsilon_m.   
\end{align}
\end{lem}
\begin{proof}
Since $ \left(\limsup_{n\rightarrow \infty}\frac{1}{A_n^\alpha}\sum_{k=0}^n A_{n-k}^{\alpha-1} \norm{b^{(m)}_k}_{\infty}^q\right)^\frac{1}{q} \leq \epsilon_m$, there exists $l_m\in \mathbb{N}$ such that  
\begin{align}\label{Eq:LimsupDominant}
\left(\frac{1}{A_n^\alpha}\sum_{k=0}^n A_{n-k}^{\alpha-1} \norm{b^{(m)}_k}_{\infty}^q\right)^\frac{1}{q} \leq 2\epsilon_m , \qquad 
\forall n\geq l_m.   
\end{align}
Set $N_m=\max\{l_1,\ldots, l_m\}+m$. Then, $N_1<N_2<\cdots <N_m<\cdots$ with $N_m\geq m$.  Let $\mathbf{d}^{(m)}=\{{d}^{(m)}_k\}_{k\geq 0}$ defined by
\begin{align*}
{d}^{(m)}_k
=\begin{cases}
b^{(m)}_k, \qquad\text{ if } k\geq N_m;\\ 
0, \qquad\quad \text{ if } k<N_m.  
\end{cases} 
\end{align*}
Note that ${d}_k^{(m)} -{b}_k^{(m)}=0$ if $k\geq N_m$, that is, $\mathbf{d}^{(m)} -\mathbf{b}^{(m)}$ has finite support.
Thus, by Lemma \ref{Lem:bk=0implynorm0}, we have
\begin{align*}
\norm{\mathbf{d}^{(m)} -\mathbf{b}^{(m)}}_{ W_q^\alpha(\M)}= 0.  
\end{align*}
Now if $n< N_m$, then $d_{k}^{(m)}=0$, for all $0\leq k\leq n$ and hence, $\left(\frac{1}{A_n^\alpha}\sum_{k=0}^n A_{n-k}^{\alpha-1} \norm{d^{(m)}_k}_{\infty}^q\right)^\frac{1}{q}=0$. On the other hand, if $n\geq  N_m$, then $n> l_m$, and by Eq. \eqref{Eq:LimsupDominant}, it follows that
\begin{align*}
\left(\frac{1}{A_n^\alpha}\sum_{k=0}^n A_{n-k}^{\alpha-1} \norm{d^{(m)}_k}_{\infty}^q\right)^\frac{1}{q}
\leq 
\left(\frac{1}{A_n^\alpha}\sum_{k=0}^n A_{n-k}^{\alpha-1} \norm{b^{(m)}_k}_{\infty}^q\right)^\frac{1}{q} \leq 2\epsilon_m.   
\end{align*}
Thus, Eq. \eqref{Eq:ApproximationByTrunction} holds. This completes the proof.
\end{proof}

\begin{thm}\label{Thm:WqAlphaComplete}
Let $\alpha>0$ and $1\leq q\leq \infty$. Then, the space $W_q^\alpha(\M)$ is complete with respect to the topology generated by the semi-norm $\norm{\cdot}_{W_q^\alpha(\M)}$.     
\end{thm}

\begin{proof}
Note that the proof for $q=\infty$ is trivial. In fact, $\norm{\cdot}_{W_\infty^\alpha(\M)}$ is a genuine norm and 
$W_\infty^\alpha(\M)$ is a Banach space. Thus, we may assume that $1\leq q<\infty$. 

It is enough to show that $\sum_{m\geq 1}\mathbf{b}^{(m)}$ converges in $W_q^\alpha(\M)$ whenever $\norm{\mathbf{b}^{(m)}}_{W_q^\alpha(\M)}\leq \epsilon_m$, for all $m\geq 1$, with $\sum_{m\geq 1}\epsilon_m<\infty$. 

\noindent Let $N_m$ and $\mathbf{d}^{(m)}$ be as in Lemma \ref{Lem:TruncationWithApproximation}. Fix $k\geq 0$. Note that if $d_k^{(m)}\neq 0$, then $k\geq N_m\geq m$. So 
\begin{align}\label{Eq:NewElemenetasLimit}
b_k=\sum_{m\geq 1} d_k^{(m)}=\sum_{m=1}^k d_k^{(m)}    
\end{align}
is a finite sum in $\M$ and hence $\mathbf{b}=\{b_k\}_{k\geq 0}\subseteq \M$ is well-defined. We show that $\mathbf{b}\in W_q^\alpha(\M)$ and 
\begin{align*}
\norm{\mathbf{b}-\sum_{m=1}^{l}\mathbf{b}^{(m)}}_{W_{q}^\alpha(\M)}\rightarrow 0 \qquad\text{ as }\quad l\rightarrow\infty.
\end{align*}
Fix $l\in\mathbb{N}$ and $n\in\mathbb{N}$ with $l\leq n$. 
Note that if $k\geq l$, we have
\begin{align*}
[b-\sum_{m=1}^{l} d^{(m)}]_{k}
=&b_k-\sum_{m=1}^{l} d_k^{(m)}
\underset{(\text{Eq. }\eqref{Eq:NewElemenetasLimit})}{=}\sum_{m=1}^{k} d_k^{(m)}-\sum_{m=1}^{l} d_k^{(m)}\\
=&\sum_{m=l+1}^{k} d_k^{(m)}
\underset{( d_k^{(m)}=0,\forall\,m\geq k)}{=}\sum_{m=l+1}^{n} d_k^{(m)}
=[\sum_{m=l+1}^{n} \mathbf{d}^{(m)}]_{k}.
\end{align*}
On the other hand, $d_k^{(m)}=0$, for all $m\geq k$. Now if $l>k$, then we have
\begin{align*}
[b-\sum_{m=1}^{l} d^{(m)}]_{k}
=b_k-\sum_{m=1}^{l} d_k^{(m)}
=\sum_{m=1}^{k} d_k^{(m)}-\sum_{m=1}^{l} d_k^{(m)}
=-\sum_{m=k+1}^{l} d_k^{(m)}=0=[\sum_{m=l+1}^{n} \mathbf{d}^{(m)}]_k
\end{align*}
In the above two cases, the coordinates with index
$k\leq n$ of $\mathbf{b}-\sum_{m=1}^{l} \mathbf{d}^{(m)}$ agree with those of the finite sum $\sum_{m=l}^{n} \mathbf{d}^{(m)}$. Thus, by the Minkowski inequality, we have
\begin{align}\label{Eq:WithOutLimitTail}
\left( \frac{1}{A_n^\alpha}\sum_{k=0}^n A_{n-k}^{\alpha-1} \norm{b_k-\sum_{m=1}^{l}d^{(m)}_k}_{\infty}^q\right)^\frac{1}{q}
=&\left( \frac{1}{A_n^\alpha}\sum_{k=0}^n A_{n-k}^{\alpha-1} \norm{\sum_{m=l+1}^{n}d^{(m)}_k}_{\infty}^q\right)^\frac{1}{q}\\
\nonumber\leq &\left( \frac{1}{A_n^\alpha}\sum_{k=0}^n A_{n-k}^{\alpha-1} \left(\sum_{m=l+1}^{n}\norm{d^{(m)}_k}_{\infty}\right)^q\right)^\frac{1}{q}\\
\nonumber\underset{(\text{Minkowski})}{\leq} &\sum_{m=l+1}^{n} \left( \frac{1}{A_n^\alpha}\sum_{k=0}^n A_{n-k}^{\alpha-1} \norm{d^{(m)}_k}_{\infty}^q\right)^\frac{1}{q}\\
\nonumber\underset{(\text{Eq. } \eqref{Eq:ApproximationByTrunction})}{\leq}& \sum_{m=l+1}^{n} 2\epsilon_m \\
\nonumber\leq& \sum_{m\geq l+1} 2\epsilon_m.  
\end{align}
Since the right-hand side of Eq. \eqref{Eq:WithOutLimitTail} is Cauchy, we have
\begin{align*}
\norm{\mathbf{b}-\sum_{m=1}^{l}\mathbf{b}^{(m)}}_{W_{q}^\alpha(\M)}
=&\left(\limsup_{n\rightarrow\infty} \frac{1}{A_n^\alpha}\sum_{k=0}^n A_{n-k}^{\alpha-1} \norm{b_k-\sum_{m=1}^{l}d^{(m)}_k}_{\infty}^q\right)^\frac{1}{q}\\
\leq &\limsup_{n\rightarrow\infty}\left( \frac{1}{A_n^\alpha}\sum_{k=0}^n A_{n-k}^{\alpha-1} \norm{b_k-\sum_{m=1}^{l}d^{(m)}_k}_{\infty}^q\right)^\frac{1}{q}\\
\leq& \sum_{m\geq l+1} 2\epsilon_m\longrightarrow 0 \text{ as }l\rightarrow\infty.
\end{align*}
Moreover,
\begin{align*}
\norm{\mathbf{b}}_{W_{q}^\alpha(\M)}    
\leq &\norm{\mathbf{b}-\sum_{m=1}^{l}\mathbf{b}^{(m)}}_{W_{q}^\alpha(\M)}+\norm{\sum_{m=1}^{l}\mathbf{b}^{(m)}}_{W_{q}^\alpha(\M)}\\
\leq & \norm{\mathbf{b}-\sum_{m=1}^{l}\mathbf{b}^{(m)}}_{W_{q}^\alpha(\M)}+\sum_{m=1}^{l}\norm{\mathbf{b}^{(m)}}_{W_{q}^\alpha(\M)}<\infty.
\end{align*}
This shows that $\mathbf{b}\in W_q^\alpha(\M)$. 
This completes the proof.
\end{proof}

\begin{rem}
\noindent $(i)$. In Eq. \eqref{Eq:NewElemenetasLimit}, every sum in $\M$ is finite. Therefore, Theorem \ref{Thm:WqAlphaComplete} holds with $\M$ replaced by any normed space with $\alpha>0$ and $1\leq q <\infty$; the only
property that we used in the proof is that $\norm{\cdot}_\infty$ is a norm. 

\noindent $(ii)$ We also have analogous properties for $W_q^\alpha(\M)$ to those of $W_q^\alpha$ described in Yoshimoto's paper \cite[Proposition 2]{Yo20}. Namely, $(1)$ for $\alpha>0$, we have $W_{q_2}^{\alpha}(\mathcal{M}) 
\subset W_{q_1}^{\alpha}(\mathcal{M})$, whenever $1\leq q_1<q_2\leq\infty$, and $(2)$
For $1\leq q\leq \infty$, we have $W_{q}^{\alpha_1}(\mathcal{M})\subset W_{q}^{\alpha_2}(\mathcal{M})$, whenever  
$0<\alpha_1<\alpha_2<\infty$.
The proofs are similar to those given by Yoshimoto \cite{Yo20}, and hence we omit the details.

\noindent $(iii)$
Let $\alpha>0$ and $1\leq q<\infty$. For $\mathbf{w} = \{w_k\}_{k \geq 0} \in W_q^{\alpha}$, denote
\begin{align*}
|\mathbf{w}|_q := \bigg(\sup_{n \geq 0}\frac{1}{A_n^{\alpha}}\sum_{k=0}^n A_{n-k}^{\alpha-1}\abs{w_k}^q \bigg)^{\frac{1}{q}}. 
\end{align*}
Then,  
\begin{align*}
\norm{\mathbf{w}}_{W_q^{\alpha}} < \infty \iff |\mathbf{w}|_q < \infty    
\end{align*}
We shall use these facts in proof of the Theorem \ref{Thm:buemAt0ForM_nbTx}.
\end{rem}

\subsection{Operator-Valued Besicovitch Weights}\label{Subsec:OpValuedBesicovitchWeight} 
In this section, we discuss on operator-valued Besicovitch weights. A similar definition can be found in \cite[Definition 5.1]{CL13}.

Let $\mathcal{U}(\M)$ denote the collection of all unitary operators in $\M$. For an operator $x\in \M$, we denote $\sigma(x)$ to be the spectrum of $x$. Denote 
\begin{align*}
    \mathcal{U}_f:=\left\{ u\in \mathcal{U}(\M): \sigma(u)\text{ is finite}\right\}. 
\end{align*}
Let $\mathcal{U}_0\subseteq \mathcal{U}(\M)$. An operator valued function $P:\mathbb{N}\rightarrow \M$ is said to be trigonometric polynomial over $\mathcal{U}_0$ if there exists some $m\in\mathbb{N}$ such that
\begin{align*}
P(k)
=\sum_{j=0}^{m} z_j u_{j}^k, \quad\quad\forall k\in\mathbb{N},
\end{align*}
where $z_j\in\mathbb{C}$ and $u_j\in\mathcal{U}_0$,  $0\leq j\leq m$.

\begin{defn}
Let $\mathcal{U}_0\subseteq \mathcal{U}(\M)$ . Given $\alpha>0$ and $1\leq q<\infty$, a sequence $\mathbf{b}=\{b_k\}_{k\geq 0}\in W_q^\alpha(\M)$ is said to be $\mathcal{U}_0(q,\alpha)$-Besicovitch, if for every $\varepsilon>0$, there exists a trigonometric polynomial $P_{\varepsilon}$ over $\mathcal{U}_0$ such that 
\begin{align*}
\left(\limsup_{n\rightarrow\infty}\frac{1}{A_{n}^\alpha}\sum_{k=0}^n A_{n-k}^{\alpha-1}\norm{b_k-P_\varepsilon(k)}^q_\infty\right)^\frac{1}{q}<\varepsilon.
\end{align*}
\end{defn}    
For $1\leq q< \infty$ and $\alpha>0$, we refer to the elements of $\mathcal{U}_0(q,\alpha)$ as $q$-Besicovitch sequences of order $\alpha$.
Note that a $\mathcal{U}_0(1,\alpha)$-Besicovitch sequence $\mathbf{b}=\{b_k\}_{k\geq 0}$ is said to be bounded if $\mathbf{b}\in W_{\infty}^\alpha(\M)\cap \mathcal{U}_0(1,\alpha)$, equivalently, ${\sup_k}\norm{b_k}_\infty<\infty$. When $\M=\mathbb{C}$, we 
simply write ${B}_q^\alpha$ instead of $\mathcal{U}_0(q,\alpha)$.

\section{Maximal Inequality and Mean Ergodic Theorem for Weighted Ergodic Averages}\label{Sec:MaximalInequalityWeightedErgodicAverages}
In this section, by using the results of \cite{HJZ26}, we prove the mean ergodic theorem and then establish the corresponding maximal inequality for weighted ergodic averages.

We begin with the following lemma.

\begin{lem}\label{Lem:Unitary}
Let $\alpha>0$. Let $\mathbf{u} = \{u^k\}_{k \geq 0}$ and $\mathbf{v} = \{v^k\}_{k \geq 0}$, where $u,v \in \mathcal{U}_f$. Then the following statements hold:
\begin{enumerate}[label=(\roman*)]
    \item If $T$ is a $DS$ operator on $\M$, then for every $x \in L_p(\M)$, $1<p<\infty$, the sequence $\{M_n(\mathbf{u}, \mathbf{v},\mathbf{1},T, \alpha)x\}_{n\geq 0}$ converges in the norm of $L_p(\M)$. Moreover, if $\M$ is finite, the same result holds for $p=1$. In the case $p=\infty$, when $\M$ is finite, the above convergence is with respect to the $w^*$-topology.
    \item If $T$ is a $DS^{+}$ operator on $\M$, then for every $x \in L_1(\M)\cap\M$, the sequence $\{M_n(\mathbf{u}, \mathbf{v},\mathbf{1},T, \alpha)x\}_{n\geq 0}$ converges $a.u.$.
\end{enumerate}
\end{lem}

\begin{proof}
Since $u,v \in \mathcal{U}_f$, there exist $\lambda_i,\mu_j\in\mathbb{T}$ and mutually orthogonal projections $P_i, Q_j\in\mathcal{P}(\M)$ for $1\leq i\leq m$, $1\leq j\leq l$, such that
\begin{align*}
u = \sum_{i =1}^m \lambda_iP_i \quad\text{ and }\quad
v = \sum_{j =1}^l \mu_jQ_j.
\end{align*}
Then it follows that 
\begin{align*}
u^k = \sum_{i =1}^m \lambda_i^k P_i \quad\text{ and }\quad
v^k = \sum_{j =1}^l \mu_j^k Q_j.
\end{align*}
For each fixed $i,j$, set $w_k=(\lambda_i \mu_j)^k$, $k\geq 0$. Then, $\mathbf{w}=\{w_k\}_{k\geq 0}\in B_1^\alpha$.

\vspace{2mm}
\noindent $(i)$. Let $x \in L_p(\M)$ with $1<p<\infty$. Note that
\begin{align*}
M_n(\mathbf{u}, \mathbf{v},\mathbf{1},T, \alpha)x 
= \sum_{i =1}^m \sum_{j =1}^l \bigg( \frac{1}{A_n^\alpha} \sum_{k =0}^n A_{n-k}^{\alpha-1} (\lambda_i \mu_j)^k T^k(P_ixQ_j)\bigg).
\end{align*}
 Since $P_ixQ_j\in L_p(\M)$, by \cite[Theorem 1.1]{HJZ26}, the sequence 
\begin{align*}
\frac{1}{A_n^\alpha} \sum_{k =0}^n A_{n-k}^{\alpha-1}(\lambda_i \mu_j)^k T^k(P_ixQ_j)
\end{align*}
converges in the norm of $L_p(\M)$. Consequently, by linearity, the sequence $\{M_n(\mathbf{u}, \mathbf{v},\mathbf{1},T, \alpha)x\}_{n\geq 0}$ converges in the norm of $L_p(\M)$. 

Similar proofs hold for $p=1, \infty$ when $\M$ is finite using \cite[Theorem 1.1]{HJZ26}.

\vspace{2mm}
\noindent $(ii)$. Let $x \in L_1(\M)\cap\M$. Note that
\begin{align*}
M_n(\mathbf{u}, \mathbf{v},\mathbf{1},T, \alpha)x 
= \sum_{i =1}^m \sum_{j =1}^l \bigg( \frac{1}{A_n^\alpha} \sum_{k =0}^n A_{n-k}^{\alpha-1} (\lambda_i \mu_j)^k T^k(P_ixQ_j)\bigg).
\end{align*}
 Since $P_ixQ_j\in L_1(\M)\cap\M$, by \cite[Theorem 1.3]{HJZ26}, the sequence 
\begin{align*}
\frac{1}{A_n^\alpha} \sum_{k =0}^n A_{n-k}^{\alpha-1} (\lambda_i \mu_j)^k T^k(P_ixQ_j)
\end{align*}
converges $a.u.$. Consequently, by linearity, the sequence $\{M_n(\mathbf{u}, \mathbf{v},\mathbf{1},T, \alpha)x\}_{n\geq 0}$ converges 
$a.u.$. This completes the proof.
\end{proof}

\begin{lem}\label{Lem:ConvergenceofTrigonometricPolyNomial}
Let $\alpha>0$. Let $\psi_1,\psi_2$ be two  trigonometric polynomials over $\mathcal{U}_f$. Then the following statements hold:
\begin{enumerate}[label=(\roman*)]
\item If $T$ is a $DS$ operator on $\M$, then for every $x \in L_p(\M)$, $1<p<\infty$, the sequence $\{M_n(\psi_1, \psi_2,\mathbf{1},T, \alpha)x\}_{n\geq 0}$ converges in the norm of $L_p(\M)$. Moreover, if $\M$ is finite, the same result holds for $p=1$. In the case $p=\infty$, when $\M$ is finite, the above convergence is with respect to the $w^*$-topology.
\item If $T$ is a $DS^{+}$ operator on $\M$, then for every $x \in L_1(\M)\cap\M$, the sequence $\{M_n(\psi_1, \psi_2,\mathbf{1},T, \alpha)x\}_{n\geq 0}$ converges $a.u.$.
\end{enumerate}
\end{lem}

\begin{proof}
Denote
\begin{align*}
\psi_1(\cdot) = \sum_{i = 0}^l z_i u_i^{(\cdot)}, \quad \psi_2(\cdot) = \sum_{j = 0}^m w_j v_j^{(\cdot)}, 
\end{align*}
where $u_i, v_j \in \mathcal{U}_f$ and $z_i, w_j \in \mathbb{C}$, $0\leq i\leq l$, $0\leq j\leq m$. Then, we have   
\begin{align*}
M_n(\psi_1, \psi_2,\mathbf{1},T, \alpha)x
= \sum_{i = 0}^l \sum_{j = 0}^m z_iw_j M_n(\mathbf{u}_i, \mathbf{v}_j,\mathbf{1},T, \alpha)x,
\end{align*}
where $\mathbf{u}_i=\{u_i^k\}_{k\geq 0}$ and $\mathbf{v}_j=\{v_j^k\}_{k\geq 0}$. The desired result now follows from Lemma \ref{Lem:Unitary}.
\end{proof}

Now we prove the mean ergodic theorem stated in Theorem \ref{Thm:ConvergenceInLpNormHongType} for weighted ergodic averages considered in Eq. \eqref{Eq:ModulatedWeightedAverage}.

\begin{proof}[\textbf{Proof of Theorem \ref{Thm:ConvergenceInLpNormHongType}}]
We consider three cases depending on $p$.

\vspace{2mm}
\noindent \textbf{Case I}:
Let $1< p<\infty$ and $x\in L_p(\M)$. To show the sequence $\{M_n(\mathbf{b},\mathbf{d},\mathbf{1},T,\alpha)x\}_{n\geq 0}$ converges in $L_p(\mathcal{M})$-norm, it is enough to show that 
the sequence $\{M_n(\mathbf{b},\mathbf{d},\mathbf{1},\alpha,T)x\}_{n\geq 0}$ is Cauchy in the norm of $L_p(\M)$. 

Suppose $\mathbf{d}=\{d_k\}_{k\geq 0}$ is bounded.
Set $C=\sup_k\norm{d_k}_\infty<\infty$.
Let $\varepsilon>0$. Since $\mathbf{b}\in \mathcal{U}_f(1,\alpha)$,  there exists a trigonometric polynomial over $\mathcal{U}_f$, 
\begin{align*}
\psi_1(\cdot) = \sum_{i=1}^m z_i u_i^{(\cdot)},  
\end{align*}
where $u_i\in \mathcal{U}_f$ and $z_i\in \mathbb{C}$, $1\leq i\leq m$, such that
\begin{align*}
\limsup_{n\rightarrow\infty}\frac{1}{A_{n}^\alpha}\sum_{k=0}^n A_{n-k}^{\alpha-1}\norm{b_k-\psi_1(k)}_\infty<\frac{\varepsilon}{16\norm{x}_{p}C}.
\end{align*}
Denote $D:=\sum_{i=1}^m\abs{z_i}$. Then, note that
\begin{align*}
\sup_{k\geq 0}\norm{\psi_1(k)}_{\infty}
\leq \sum_{i=1}^{m}\abs{z_i}
= D.
\end{align*}
Again, since $\mathbf{d}\in \mathcal{U}_f(1,\alpha)$,  there exists a trigonometric polynomial over $\mathcal{U}_f$,
\begin{align*}
\psi_2(\cdot) = \sum_{j=1}^l w_j v_j^{(\cdot)},    
\end{align*}
where $v_j\in\mathcal{U}_f$ and $w_j\in \mathbb{C}$, $1\leq j\leq l$, such that
\begin{align*}
\limsup_{n\rightarrow\infty}\frac{1}{A_{n}^\alpha}\sum_{k=0}^n A_{n-k}^{\alpha-1}\norm{d_k-\psi_2(k)}_\infty<\frac{\varepsilon}{16\norm{x}_p D}.
\end{align*}
Thus, by definition of $\limsup$, there exists a natural number $n_0\in\mathbb{N}$ such that for all $n\geq n_0$, we have
\begin{align}\label{Eq:LimSupEstimate}
\frac{1}{A_{n}^\alpha}\sum_{k=0}^n A_{n-k}^{\alpha-1}\norm{b_k-\psi_1(k)}_\infty<\frac{\varepsilon}{8\norm{x}_{p}C}\text{ and }  \frac{1}{A_{n}^\alpha}\sum_{k=0}^n A_{n-k}^{\alpha-1}\norm{d_k-\psi_2(k)}_\infty<\frac{\varepsilon}{8\norm{x}_p D}.
\end{align}
Note that
\begin{align}\label{Eq:CauchyInNorm}
\norm{M_n(\mathbf{b},\mathbf{d},\mathbf{1},\alpha,T)x-M_m(\mathbf{b},\mathbf{d},\mathbf{1},\alpha,T)x}_p
\leq \mathrm{I}+\mathrm{II}+\mathrm{III}+\mathrm{IV}+\mathrm{V},
\end{align}
where
\begin{align*}
\mathrm{I}&=\norm{M_n(\mathbf{b},\mathbf{d},\mathbf{1},\alpha,T)x-M_n(\psi_1,\mathbf{d},\mathbf{1},\alpha,T)x}_p\\
\mathrm{II}&=\norm{M_n(\psi_1,\mathbf{d},\mathbf{1},\alpha,T)x-M_n(\psi_1,\psi_2,\mathbf{1},\alpha,T)x}_p\\
\mathrm{III}&=\norm{M_n(\psi_1,\psi_2,\mathbf{1},\alpha,T)x-M_m(\psi_1,\psi_2,\mathbf{1},\alpha,T)x}_p\\
\mathrm{IV}&=\norm{M_m(\psi_1,\psi_2,\mathbf{1},\alpha,T)x-M_m(\psi_1,\mathbf{d},\mathbf{1},\alpha,T)x}_p, \text{ and }\\
\mathrm{V}&=\norm{M_m(\psi_1,\mathbf{d},\mathbf{1},\alpha,T)x-M_m(\mathbf{b},\mathbf{d},\mathbf{1},\alpha,T)x}_p.
\end{align*}
Since $T$ is $DS$, for $n\geq n_0$, we have
\begin{align}\label{Eq:IstEstimate}
\mathrm{I}
\leq 
\frac{1}{A_n^{\alpha}}\sum_{k =0}^n A_{n-k}^{\alpha -1} \norm{T^k((b_k-\psi_1(k)) x d_k)}_p
\leq &\frac{1}{A_n^{\alpha}}\sum_{k =0}^n A_{n-k}^{\alpha -1} \norm{(b_k-\psi_1(k)) x d_k)}_p\\
\nonumber\leq  &\frac{1}{A_n^{\alpha}}\sum_{k =0}^n A_{n-k}^{\alpha -1} \norm{b_k-\psi_1(k)}_\infty\norm{x}_{p}\norm{ d_k}_\infty\\
\nonumber\leq &\left(\norm{x}_p C\right)\frac{1}{A_n^{\alpha}}\sum_{k =0}^n A_{n-k}^{\alpha -1} \norm{b_k-\psi_1(k)}_\infty\\
\nonumber\leq &\frac{\varepsilon}{8},
\end{align}
where the last inequality follows from Eq. \eqref{Eq:LimSupEstimate}. Similarly, since $T$ is $DS$, for all $n\geq n_0$, we have
\begin{align}\label{Eq:IIndEstimate}
 \mathrm{II}
 \leq 
\frac{1}{A_n^{\alpha}}\sum_{k =0}^n A_{n-k}^{\alpha -1} \norm{T^k(\psi_1(k)x(d_k-\psi_2(k)))}_p
\leq &\frac{1}{A_n^{\alpha}}\sum_{k =0}^n A_{n-k}^{\alpha -1} \norm{\psi_1(k)x(d_k-\psi_2(k))}_p\\
\nonumber\leq &\frac{1}{A_n^{\alpha}}\sum_{k =0}^n A_{n-k}^{\alpha -1} \norm{\psi_1(k)}_{\infty}\norm{x}_p\norm{d_k-\psi_2(k)}_\infty\\
\nonumber\leq &\left(D\norm{x}_p\right)\frac{1}{A_n^{\alpha}}\sum_{k =0}^n A_{n-k}^{\alpha -1} \norm{d_k-\psi_2(k)}_\infty\\
\nonumber\leq &\frac{\varepsilon}{8},
\end{align}
where the last inequality follows from Eq. \eqref{Eq:LimSupEstimate}. 
Again, by Lemma \ref{Lem:ConvergenceofTrigonometricPolyNomial}, there exists $n_1\in\mathbb{N}$, such that for all $m,n\geq n_1$, we have
\begin{align}\label{Eq:IIIrdEstimate}
\mathrm{III}
=\norm{M_n(\psi_1,\psi_2,\mathbf{1},\alpha,T)x-M_m(\psi_1,\psi_2,\mathbf{1},\alpha,T)x}_p
\leq \frac{\varepsilon}{8}.
\end{align}
Therefore, for all $m,n\geq n_2=\max(n_0,n_1)$, from Eqs. \eqref{Eq:CauchyInNorm}, \eqref{Eq:IstEstimate}, \eqref{Eq:IIndEstimate} and \eqref{Eq:IIIrdEstimate}, we have
\begin{align*}
\norm{M_n(\mathbf{b},\mathbf{d},\mathbf{1},\alpha,T)x-M_m(\mathbf{b},\mathbf{d},\mathbf{1},\alpha,T)x}_p
\leq \mathrm{I}+\mathrm{II}+\mathrm{III}+\mathrm{IV}+\mathrm{V}
\leq  \frac{5\varepsilon}{8}
<\varepsilon.
\end{align*}
Thus, the sequence $\{M_n(\mathbf{b},\mathbf{d},\mathbf{1},\alpha,T)x\}_{n\geq 0}$ is Cauchy in the norm of $L_p(\M)$. 

A similar argument applies when $\mathbf{b}$ is bounded. This completes the proof for $1<p<\infty$.

\vspace{2mm}
\noindent\textbf{Case II}: 
When $\M$ is finite and $p=1$, the result follows similarly to Case I by applying Lemma \ref{Lem:ConvergenceofTrigonometricPolyNomial}.

Now we deal with the case when $\M$ is finite and $p=\infty$. We know that $\M = (L_1(\M))^*$ and this duality is given by $\M\ni x\mapsto \tau_x\in (L_1(\M))^*$, where $\tau_x(y)=\tau(xy)$, $y\in L_1(\M)$.
Now given a $DS$ operator $T$, by considering $T$ as a map from $L_1(\M)$ to $L_1(\M)$, we can take the dual map $T^*:\M\rightarrow \M$ given by
\begin{align*}
\tau((T^*x)y)=\tau(x(Ty)),\qquad\forall x\in \M,\quad y\in L_1(\M).    
\end{align*}
It is easy to check that $T^*$ is also a $DS$ operator. 

Fix $x\in \M$. To show that $\{M_n(\mathbf{b},\mathbf{d},\mathbf{1},\alpha,T)x\}_{n\geq 0}$ converges in weak$^*$-topology, we need to prove that the sequence $\{\tau\left(y M_n(\mathbf{b},\mathbf{d},\mathbf{1},\alpha,T)x\right)\}_{n\geq 0}$ is convergent for all $y\in L_1(\M)$. Let $\varepsilon>0$ and $y\in L_1(\M)$. Since $\mathbf{b}, \mathbf{d}\in \mathcal{U}_f(1,\alpha)$,  there exist trigonometric polynomials $\psi_1,\psi_2$ over $\mathcal{U}_f$ such that 
\begin{align*}
\limsup_{n\rightarrow\infty}\frac{1}{A_{n}^\alpha}\sum_{k=0}^n A_{n-k}^{\alpha-1}\norm{b_k-\psi_1(k)}_\infty<\frac{\varepsilon}{16\norm{x}_{\infty}C\norm{y}_1},
\end{align*}
and
\begin{align*}
    \limsup_{n\rightarrow\infty}\frac{1}{A_{n}^\alpha}\sum_{k=0}^n A_{n-k}^{\alpha-1}\norm{d_k-\psi_2(k)}_\infty<\frac{\varepsilon}{16\norm{x}_{\infty}D\norm{y}_1},
\end{align*}
where $D$ is taken similar as in Case $\mathrm{I}$.
Similar to Eq. \eqref{Eq:LimSupEstimate}, there exists a natural number $n_0\in\mathbb{N}$ such that for all $n\geq n_0$, we have
\begin{align}\label{Eq:LimSupEstimate1}
\frac{1}{A_{n}^\alpha}\sum_{k=0}^n A_{n-k}^{\alpha-1}\norm{b_k-\psi_1(k)}_\infty
<\frac{\varepsilon}{8\norm{x}_{\infty}C\norm{y}_1}
\text{ and }  \frac{1}{A_{n}^\alpha}\sum_{k=0}^n A_{n-k}^{\alpha-1}\norm{d_k-\psi_2(k)}_\infty<\frac{\varepsilon}{8\norm{x}_{\infty}D\norm{y}_1}.
\end{align}
Since $x\in \M\subseteq M\cap L_1(\M)$, by Lemma \ref{Lem:ConvergenceofTrigonometricPolyNomial}, the sequence $\{M_n(\psi_1,\psi_2,\mathbf{1}, T,\alpha)x\}_{n\geq 0}$ converges in $w^*$-topology. Consequently, the sequence $\{\tau(yM_n(\psi_1,\psi_2,\mathbf{1},\alpha,T)x)\}_{n\geq 0}$ is Cauchy.
Now, for $n\geq n_0$, we have the following estimate:
\begin{align*}
   & \left|\tau(y M_n(\mathbf{b},\mathbf{d},\mathbf{1},\alpha,T)x) - \tau(yM_n(\psi_1,\psi_2,\mathbf{1},\alpha,T)x)\right| \\
 \leq&  \left|\tau(y M_n(\mathbf{b},\mathbf{d},\mathbf{1},\alpha,T)x) -\tau(y M_n(\psi_1,\mathbf{d},\mathbf{1},\alpha,T)x) \right| \\
  & \qquad \qquad +\left|\tau(y M_n(\psi_1,\mathbf{d},\mathbf{1},\alpha,T)x) -\tau(yM_n(\psi_1,\psi_2,\mathbf{1},\alpha,T)x)\right| \\
 =& \mathrm{I}+\mathrm{II},
\end{align*}
where
\begin{align*}
    \mathrm{I} &= \left|\tau(y M_n(\mathbf{b},\mathbf{d},\mathbf{1},\alpha,T)x) -\tau(y M_n(\psi_1,\mathbf{d},\mathbf{1},\alpha,T)x) \right| \\
     \leq& \frac{1}{A_n^{\alpha}}\sum_{k=0}^n A_{n-k}^{\alpha-1}\abs{\tau(yT^k((b_k-\psi_1(k)) xd_k))}\\
    =& \frac{1}{A_n^{\alpha}}\sum_{k=0}^n A_{n-k}^{\alpha-1}\abs{\tau((T^*)^k(y)(b_k-\psi_1(k)) xd_k)}\\
    \leq & \frac{1}{A_n^{\alpha}}\sum_{k=0}^n A_{n-k}^{\alpha-1}\norm{(T^*)^k(y)}_1\norm{b_k-\psi_1(k)}_\infty \norm{x}_\infty\norm{d_k}_\infty\\
   \leq & \left(\norm{x}_\infty C\norm{y}_1\right)\frac{1}{A_n^{\alpha}}\sum_{k=0}^n A_{n-k}^{\alpha-1}\norm{b_k-\psi_1(k)}_\infty\\
   \underset{(\text{Eq. \eqref{Eq:LimSupEstimate1}})}{<}& \frac{\varepsilon}{8}.
\end{align*}
Similarly, it can be shown that $\mathrm{II}< \frac{\varepsilon}{8}$. 
Therefore, $\{\tau(yM_n(\mathbf{b}, \mathbf{d},\mathbf{1},T,\alpha))\}_{n\geq 0}$ is Cauchy and hence convergent. This completes the proof. 
\end{proof}

We now prove the maximal inequality as given in Theorem \ref{Thm:MaximalInequalityForVectorWeightedSequenceSpace}, which is one of the main results of this paper.


\begin{proof}[\textbf{Proof of Theorem \ref{Thm:MaximalInequalityForVectorWeightedSequenceSpace}}]
Suppose $\mathbf{d}=\{d_k\}_{k\geq 0}$ is bounded.
To prove this result, it is enough to assume that $x\in L_p(\M)_+$.

For $k\geq 0$, define 
\begin{align*}
w_k=\norm{b_k d_k}_\infty
\qquad\text{ and }\qquad 
\widetilde{b}_k
=\begin{cases}
\frac{b_k d_k}{\norm{b_k d_k}_\infty },\qquad &\text{ if }    \norm{b_k d_k}_\infty\neq 0;\\
 0, \qquad &\text{ otherwise}. 
\end{cases}
\end{align*}
Then $\mathbf{\widetilde{b}}:=\{\widetilde{b}_k\}_{k\geq 0}$ is a $\mathcal{Z}(M)$-valued uniformly bounded sequence bounded by $1$ and $\mathbf{w}=\{w_k\}_{k\geq 0}$ is a positive element in $W_q^\alpha$. Indeed, $\mathbf{w}\in W_q^\alpha$ follows from the assumption $\mathbf{b}\in W_q^\alpha(\mathcal{Z}(\mathcal{M}))$ and $\mathbf{d}$ is bounded. 
Consequently, we have
\begin{align*}
M_n(\mathbf{\widetilde{b}},\mathbf{1}, \mathbf{w}, T, \alpha)x
=M_n(\mathbf{b},\mathbf{d}, \mathbf{1}, T, \alpha)x.
\end{align*}
Thus, to prove the results in the statement, it is enough to show $(i)$, $(ii)$ and $(iii)$ hold for the  averages $M_n(\mathbf{b},\mathbf{1}, \mathbf{w}, T, \alpha)$, where $\mathbf{b}=\{b_k\}_{k\geq 0}$ is a
$\mathcal{Z}(M)$-valued uniformly bounded sequence bounded by $1$ and 
$\mathbf{w}=\{w_k\}_{k\geq 0}$ is a positive element in $W_q^\alpha$. 

Since $\sup_{k\geq 0}\norm{b_k}_{\infty}\leq 1$, we have $\RE(b_k)+1, \IM(b_k)+1\in \mathcal{Z}(\M)$ and
\begin{align*}
0\leq \RE(b_k)+1\leq 2 \quad\text{ and } \quad 0\leq \IM(b_k)+1\leq 2, \qquad\forall k\geq 0. 
\end{align*}
Consequently, it follows from Remark \ref{Rem:Commutivityofoperators}, that 
\begin{align*}
0\leq \left(\RE(b_k)+1\right)x\leq 2x \quad\text{ and }\quad 0\leq \left(\IM(b_k)+1\right)x\leq 2x, \qquad\forall k\geq 0.    
\end{align*}
By linearity of $T$, we can write 
\begin{align*}
M_n(\mathbf{b},\mathbf{1}, \mathbf{w}, T, \alpha) x
=&M_n(\RE(\mathbf{b}+\mathbf{1}),\mathbf{1}, \mathbf{w}, T, \alpha) x 
+  i M_n(\IM(\mathbf{b}+\mathbf{1}),\mathbf{1}, \mathbf{w}, T, \alpha) x \\
&\qquad-(1+i)M_n(\mathbf{1},\mathbf{1}, \mathbf{w}, T, \alpha) x.
\end{align*}
Since $T\in DS^+$, we have 
\begin{align*}
0\leq M_n(\RE(\mathbf{b}+\mathbf{1}),\mathbf{1}, \mathbf{w}, T, \alpha) x 
\leq 2 M_n(\mathbf{1},\mathbf{1}, \mathbf{w}, T, \alpha) x
\end{align*}
and 
\begin{align*}
0\leq  M_n(\IM(\mathbf{b}+\mathbf{1}),\mathbf{1}, \mathbf{w}, T, \alpha) x 
 \leq 2M_n(\mathbf{1},\mathbf{1}, \mathbf{w}, T, \alpha) x.    
\end{align*}
By using the triangle inequality, we have
\begin{align*}
\norm{\left(M_n(\mathbf{b}, \mathbf{1}, \mathbf{w},T, \alpha)x\right)_{n\geq 0}}_{L_p(\mathcal{M}; \ell_{\infty})}
\leq 6 \norm{\left(M_n(\mathbf{1}, \mathbf{1}, \mathbf{w},T, \alpha)x\right)_{n\geq 0}}_{L_p(\mathcal{M}; \ell_{\infty})}.
\end{align*}
The conclusions in $(i)$ now follows from \cite[Theorem 1.2]{HJZ26}.  
The same conclusion holds when $\mathbf{b}=\{b_k\}_{k\geq 0}$ is bounded, by an analogous argument.

The same argument applies to $(ii)$ and $(iii)$ by invoking \cite[Theorem 1.2]{HJZ26}; hence, we omit the details.
This completes the proof.
\end{proof}

The following weak-type inequality will be useful in our analysis. Its proof is similar to that of \cite[Corollary 4.4]{HJZ26}. More precisely, in the following theorem, we establish a weak-type inequality for the $\mathcal{Z}(M)$-valued weighted averages considered in Eq. \eqref{Eq:ModulatedWeightedAverage}.

\begin{thm}\label{Thm:OneSidedWeakTypeInequality}
For $\alpha>0$ and $1< q\leq\infty$, let $\mathbf{b}=\{b_k\}_{k\geq 0}\in W_{q}^\alpha(\mathcal{Z}(\M))$. Let $q^\prime$ denote the conjugate index of $q$, that is, 
$\frac{1}{q}+\frac{1}{q^\prime}=1$. Let $T$ be a $DS^{+}$ operator on $\M$. Then the following statements hold:
\begin{enumerate}[label=(\roman*)]
\item For every $x\in L_p(\M)$ with $\max \{\frac{q^\prime}{\alpha}, q^\prime\}<p\leq\infty$ and for any $\varepsilon>0$, there exists a projection $e\in\mathcal{P}(\M)$ such that
\begin{align*}
\tau(e^\perp)\leq 4\left(C_{\alpha, p, q}\frac{\norm{x}_p}{\varepsilon}\right)^p 
\quad\text{ and }\quad 
\sup_{n\geq 0}\norm{e \left(M_n(\mathbf{b},\mathbf{1}, \mathbf{1}, T, \alpha) x\right) e}_\infty \leq 24\varepsilon,
\end{align*}
where $C_{\alpha, p, q}$ is a positive constant which depends only on the parameters $\alpha, p$ and $q$. Moreover, for every $x\in L_p(\M)$ with $\max\{\frac{2q^\prime}{\alpha}, 2q^\prime\}<p\leq \infty$ and for any $\varepsilon>0$, there exists a projection $e\in\mathcal{P}(\M)$ such that
\begin{align*}
\tau(e^\perp)\leq 4\left(\sqrt{C_{\alpha, \frac{p}{2}, q}}\frac{\norm{x}_p}{\varepsilon}\right)^p 
\quad\text{ and }\quad 
\sup_{n\geq 0}\norm{ \left(M_n(\mathbf{b},\mathbf{1}, \mathbf{1}, T, \alpha) x\right) e}_\infty \leq 24\varepsilon.
\end{align*} 
\item If $1<q\leq \infty$, and $\alpha\geq 1$, then for every $x\in L_p(\M)$ with $p=q^\prime$ $($resp. $p=2q^\prime)$, for any $\varepsilon>0$, there exists a projection $e\in\mathcal{P}(\M)$ such that 
\begin{align*}
 \tau(e^\perp)\leq \frac{64\norm{x}_p^p}{\varepsilon^p} \quad \text{ and }\quad &\sup_{n\geq 0}\norm{e \left(M_n(\mathbf{b},\mathbf{1}, \mathbf{1}, T, \alpha) x\right) e}_\infty
 \leq 24\varepsilon \\
 (\text{resp.  }&\sup_{n\geq 0}\norm{ \left(M_n(\mathbf{b},\mathbf{1}, \mathbf{1}, T, \alpha) x\right) e}_\infty\lesssim_\alpha \varepsilon).
\end{align*}
\end{enumerate}
\end{thm}

\begin{proof}
The proof of this result is similar to the proof of  \cite[Corollary 4.4]{HJZ26} by replacing the role of \cite[Theorem 1.2]{HJZ26} by Theorem \ref{Thm:MaximalInequalityForVectorWeightedSequenceSpace}. Thus, we omit the details. 
\end{proof}

\section{Subsequential Weighted Ergodic Average on Noncommutative Orlicz Spaces}\label{Sec:MainResults}
This is the final section of this paper. In this section, we study the $b.a.u.$ $($resp. $a.u.)$ convergence in a noncommutative Orlicz space of weighted ergodic averages considered in Eqs. \eqref{Eq:ModulatedWeightedAverage} and \eqref{Eq:ModulatedWeightedAverage2}. Some of our techniques are inspired by \c C\"omez and Litvinov \cite{CL13}. The $p$-convexity of the underlying Orlicz function plays a crucial role in our approach. 
We begin with the following theorem, which shows that the aforesaid weighted ergodic averages are $b.u.e.m.$ at $0$ in noncommutative Orlicz spaces.

\begin{thm}\label{Thm:buemAt0ForM_nbTx}
For $\alpha>0$ and $1< q\leq\infty$, let $\mathbf{b}=\{b_k\}_{k\geq 0}\in W_{q}^\alpha(\mathcal{Z}(\M))$. Let $q^\prime$ denote the conjugate index of $q$, that is, $\frac{1}{q}+\frac{1}{q^\prime}=1$.
Let $T$ be a $DS^{+}$ operator on $\M$. Then the following statements hold:

\noindent $($a$)$ The sequence $\{M_n(\mathbf{b},\mathbf{1}, \mathbf{1}, T, \alpha)\}_{n\geq 0}$ is $b.u.e.m.$ $($resp. $u.e.m.)$ at $0$ on $L_\Phi(\M)$ in the following cases:
\begin{enumerate}[label=(\roman*)]
\item For $\alpha>0$, $\Phi$ is a $p$-convex Orlicz  function with $\max \{\frac{q^\prime}{\alpha}, q^\prime\}< p<\infty$ $($resp. $\max \{\frac{2q^\prime}{\alpha}, 2q^\prime\}< p<\infty)$. 
\item For $\alpha \geq 1$, $\Phi$ is a $p$-convex Orlicz  function with $p=q^\prime<\infty$ $($resp. $p=2q^\prime<\infty)$.
\end{enumerate}
Moreover, if $\widetilde{\Phi}(t)=\Phi(\sqrt{t})$ is $p$-convex, then $b.u.e.m.$ at $0$ on $L_\Phi(\M)$ will be replaced by $u.e.m.$ at $0$ on $L_\Phi(\M)$. 

\noindent $($b$)$ If $\alpha\geq 1$, and 
$\mathbf{k}=\{k_j\}_{j\geq 0}$ is a strictly increasing sequence of natural numbers with lower density $d>0$, then the conclusion of $($a$)$ also holds along the subsequence $\{M_n^{\mathbf{k}}(\mathbf{b},\mathbf{1}, \mathbf{1}, T, \alpha)\}_{n\geq 0}$.
\end{thm}
\begin{proof}
We prove the result only for the $b.u.e.m.$ case with $\alpha>0$ and $\Phi$ is a $p$-convex Orlicz  function, where $\max \{\frac{q^\prime}{\alpha}, q^\prime\}< p<\infty$. The other $b.u.e.m.$ cases and $u.e.m.$ cases can be handled similarly. 

As before, in the proof of Theorem \ref{Thm:MaximalInequalityForVectorWeightedSequenceSpace}, define 
\begin{align*}
\widetilde{b}_k=\frac{b_k}{\norm{b_k}_\infty} \quad\text{ and }\quad w_k=\norm{b_k}_\infty, \qquad\qquad k\geq 0.    
\end{align*}
Then $\mathbf{\widetilde{b}}:=\{\widetilde{b}_k\}_{k\geq 0}$ is a $\mathcal{Z}(M)$-valued sequence with $\norm{\widetilde{b}_k}_\infty=1$ for all $k\geq 0$, and $\mathbf{w}=\{w_k\}_{k\geq 0}$ is a scalar sequence in $W_q^\alpha$ such that
\begin{align*}
M_n(\mathbf{\widetilde{b}},\mathbf{1}, \mathbf{w}, T, \alpha)
=M_n(\mathbf{b},\mathbf{1}, \mathbf{1}, T, \alpha).
\end{align*}
Thus, without loss of generality, we can assume that $\mathbf{b}=\{b_k\}_{k\geq 0}$ is a $\mathcal{Z}(\M)$-valued bounded sequence with $\norm{b_k}_\infty\leq 1$ for all $k\geq 0$, and $\mathbf{w}=\{w_k\}_{k\geq 0}$ is a positive scalar sequence in $W_q^\alpha$, and we show that the sequence
$\{M_n(\mathbf{b},\mathbf{1}, \mathbf{w}, T, \alpha)\}_{n\geq 0}$ $($resp. $\{M_n^{\mathbf{k}}(\mathbf{b},\mathbf{1}, \mathbf{w}, T, \alpha)\}_{n\geq 0})$ is $b.u.e.m.$ at $0$ on $L_\Phi(\M)$.

\vspace{2mm}

\noindent $($a$)$. 
Fix $\varepsilon,\delta>0$. Let $\abs{\mathbf{w}}_q := \abs{\mathbf{w}}_{W_q^{\alpha}}$. 
By Lemma \ref{Lem:pConvex}, there exists $t>0$ such that 
\begin{align*}
t\Phi^\frac{1}{p}(u)\geq u, \quad\text{ whenever }u\geq \frac{\delta}{2\abs{\mathbf{w}}_q}.
\end{align*}

Let $\gamma>0$ such that $\gamma<\min\{1, \frac{\varepsilon \delta^p}{4(48t C_{\alpha,p,q})^p}\}$, where $C_{\alpha,p,q}$ is the positive constant appearing in Theorem \ref{Thm:OneSidedWeakTypeInequality}$($i$)$. 
Suppose $x\in L_\Phi(\M)$ such that $\norm{x}_\Phi<\gamma$. We show that there exists a projection $e\in\mathcal{P}(\M)$ such that
\begin{align}\label{Eq:VectorValuedWeakType}
\tau(e^\perp)<\varepsilon\quad\text{ and }\quad \sup_{n\geq 0}\norm{e \left(M_n(\mathbf{b},\mathbf{1}, \mathbf{w}, T, \alpha) x\right)e}_\infty \leq \delta.  
\end{align}

First we show that Eq. \eqref{Eq:VectorValuedWeakType} holds for $x\in L_\Phi(\M)_+$ with $\norm{x}_\Phi<\gamma$ and  $b_k\geq 0$ for every $k\geq 0$. Let $x=\int_0^\infty \lambda\, de_\lambda$  denote the spectral decomposition of $x$. In the proof of Proposition  \ref{Prop:ToMakeSmallElementInOrliczSpace}, we rewrite Eq. \eqref{Eq:ToMakeSmallElementInOrliczSpace} as follows:
\begin{align*}
x
=\int_0^{\frac{\delta}{2\abs{\mathbf{w}}_q}} \lambda\, de_\lambda \,+\, \int_{\frac{\delta}{2\abs{\mathbf{w}}_q}}^{\infty} \lambda\, de_\lambda 
\leq \int_0^{\frac{\delta}{2\abs{\mathbf{w}}_q}} \lambda\, de_\lambda \,+\, t\int_{\frac{\delta}{2\abs{\mathbf{w}}_q}}^{\infty} \Phi^\frac{1}{p}(\lambda)\, de_\lambda 
\leq x_\delta +ty,
\end{align*}
where $x_\delta=\int_0^{\frac{\delta}{2\abs{\mathbf{w}}_q}} \lambda\, de_\lambda$ and $y=\Phi^\frac{1}{p}(x)$.

Since $T\in DS^+$, we have
\begin{align}\label{Eq:MnxdeltaEstimate}
M_n(\mathbf{b}, \mathbf{1}, \mathbf{w}, T, \alpha)x_\delta 
=&\frac{1}{A_n^{\alpha}} \sum_{k=0}^n A_{n-k}^{\alpha-1}w_kT^k(b_k x_\delta) \\
\nonumber=& \frac{1}{A_n^{\alpha}} \sum_{k=0}^n (A_{n-k}^{\alpha-1})^{\frac{1}{q}}w_k(A_{n-k}^{\alpha-1})^{\frac{1}{q^\prime}}T^k(b_k x_\delta) \\
\nonumber\underset{(\text{Eq. \eqref{Eq:HolderInequality}})}{\leq}& \bigg(\frac{1}{A_n^{\alpha}}\sum_{k=0}^n A_{n-k}^{\alpha-1}w_k^q \bigg)^{\frac{1}{q}}\bigg(\frac{1}{A_n^{\alpha}}\sum_{k=0}^n A_{n-k}^{\alpha-1}T^k((b_k x_\delta)^{q'}) \bigg)^{\frac{1}{q^\prime}} \\
\nonumber\leq &\abs{\mathbf{w}}_q \bigg(\frac{1}{A_n^{\alpha}}\sum_{k=0}^n A_{n-k}^{\alpha-1}T^k((b_k x_\delta)^{q^\prime}) \bigg)^{\frac{1}{q^\prime}}.
\end{align}
By Remark \ref{Rem:Commutivityofoperators} and using the estimate $\norm{x_\delta}_\infty\leq \frac{\delta}{2\abs{\mathbf{w}}_q}$, we have
\begin{align}\label{Eq:OpNorm}
(b_k x_\delta)^{q^\prime}
\leq (\norm{b_k}_\infty\norm{x_\delta}_\infty)^{q^\prime} 1
\leq (\norm{x_\delta}_\infty)^{q^\prime}1
\leq \left(\frac{\delta}{2\abs{\mathbf{w}}_q}\right)^{q^\prime} 1. 
\end{align}
Since $T$ is $DS^+$ operator, we have $T(1)\leq 1$. Therefore, by Eq. \eqref{Eq:MnxdeltaEstimate}, we have the following inequality:
\begin{align}\label{Eq:NormMnxDeltainfty}
\norm{M_n(\mathbf{b}, \mathbf{1}, \mathbf{w}, T,\alpha)x_\delta}_\infty 
&\leq \abs{\mathbf{w}}_q \norm{\bigg(\frac{1}{A_n^{\alpha}}\sum_{k=0}^n A_{n-k}^{\alpha-1}T^k((b_k x_\delta)^{q^\prime}) \bigg)^{\frac{1}{q^\prime}}}_{\infty} \\
\nonumber&\underset{(\text{Eq. \eqref{Eq:OpNorm}})}{\leq} \abs{\mathbf{w}}_q \frac{\delta}{2\abs{\mathbf{w}}_q} \norm{\bigg(\frac{1}{A_n^{\alpha}}\sum_{k=0}^n A_{n-k}^{\alpha-1}T^k(1) \bigg)^{\frac{1}{q^\prime}}}_{\infty} \\
\nonumber&\leq  \frac{\delta}{2} \norm{\bigg(\frac{1}{A_n^{\alpha}}\sum_{k=0}^n A_{n-k}^{\alpha-1} \bigg)^{\frac{1}{q^\prime}}}_{\infty}\\  
\nonumber&=\frac{\delta}{2}.
\end{align}
On the other hand, since $\norm{x}_\Phi<\gamma<1$, it follows from the Proposition \ref{Prop:ToMakeSmallElementInOrliczSpace}, that $\Phi^\frac{1}{p}(x)\in L_p(\M)$ and 
\begin{align*}
 \norm{\Phi^\frac{1}{p}(x)}_p^p\leq \norm{x}_\Phi<1.   
\end{align*}
Therefore, by Theorem \ref{Thm:OneSidedWeakTypeInequality}(i), for $\Phi^\frac{1}{p}(x)\in L_p(\M)$ and $\nu=\frac{\delta}{48t}$, there exists a projection $e\in\mathcal{P}(\M)$ such that 
\begin{align}\label{Eq:WeakTypePhi}
 \tau(e^\perp)\leq 4\left(\frac{C_{\alpha, p,q}\norm{\Phi^\frac{1}{p}(x)}_p}{\nu}\right)^p \quad\text{ and }\quad\sup_{n\geq 0}\norm{e \left(M_n(\mathbf{b},\mathbf{1}, \mathbf{w}, T, \alpha) \Phi^\frac{1}{p}(x)\right) e}_\infty
 \leq 24\nu.
\end{align}
Consequently, we have
\begin{align*}
\tau(e^\perp)\leq 4\left(\frac{C_{\alpha, p,q}\norm{\Phi^\frac{1}{p}(x)}_p}{\nu}\right)^p 
\leq 4\norm{\Phi^\frac{1}{p}(x)}_p^p\left(\frac{C_{\alpha,p,q}}{\nu}\right)^p 
\leq 4\norm{x}_\Phi\left(\frac{C_{\alpha,p,q}}{\frac{\delta}{48t}}\right)^p
\leq \varepsilon,   
\end{align*}
and by triangle inequality, 
\begin{align*}
 &\sup_{n\geq 0}\norm{e \left(M_n(\mathbf{b},\mathbf{1}, \mathbf{w}, T, \alpha) x\right) e}_\infty\\
 \leq &  
  \sup_{n\geq 0}\norm{e \left(M_n(\mathbf{b},\mathbf{1}, \mathbf{w}, T, \alpha) x_\delta\right) e}_\infty
  + 
t\sup_{n\geq 0}\norm{e \left(M_n(\mathbf{b},\mathbf{1}, \mathbf{w}, T, \alpha) \Phi^\frac{1}{p}(x)\right) e}_\infty\\
\underset{(\text{Eq. \eqref{Eq:WeakTypePhi}})}{\leq} &   
  \sup_{n\geq 0}\norm{ \left(M_n(\mathbf{b},\mathbf{1}, \mathbf{w}, T, \alpha) x_\delta\right)}_\infty
  + 
t(24\nu)\\
\underset{(\text{Eq. \eqref{Eq:NormMnxDeltainfty}})}{ \leq} & \frac{\delta}{2}+\frac{\delta}{2}\\
 \leq & \delta. 
\end{align*}

Recall that every $x\in L_\Phi(\M)$ can be written as a linear combination of four positive elements. More precisely, there exist  $x_j\in L_\Phi(\M)_+$, $1\leq j\leq 4$, such that 
\begin{align*}
x=x_1-x_2+i(x_3-x_4)\qquad\text{ and }\qquad \norm{x_j}_\Phi\leq \norm{x}_\Phi,\qquad\text{ for }  1\leq j\leq 4.  
\end{align*}
By Eq.  \eqref{Eq:VectorValuedWeakType}, taking $\widetilde{\varepsilon}=\frac{\varepsilon}{4}$ and $\widetilde{\delta}=\frac{\delta}{4}$, there exists $\widetilde{\gamma}>0$ such that for $x_j\in L_\Phi(\M)_+$, $1\leq j\leq 4$, with $\norm{x_j}_\Phi<\widetilde{\gamma}$, there exists a projection $e_j\in\mathcal{P}(\M)$ satisfying
\begin{align*}
\tau(e_j^\perp)\leq \frac{\varepsilon}{4} \quad \text{and} \quad \sup_{n\geq 0}\norm{e_j \left(M_n(\mathbf{b},\mathbf{1}, \mathbf{w}, T, \alpha)x_j\right)e_j}_\infty \leq \frac{\delta}{4}.    
\end{align*}
Let $e=\wedge_{j=1}^4 e_j$. 
Then,
\begin{align*}
\tau(e^\perp)
\leq \sum_{j=1}^4\tau(e_j^\perp)
\leq \varepsilon.
\end{align*}
Moreover, since $e\leq e_j$, for every $j$ and for $x\in L_\Phi(\M)$ with $\norm{x}_\Phi<\widetilde{\gamma}$, we have 
\begin{align*}
 \sup_{n\geq 0}\norm{e \left(M_n(\mathbf{b},\mathbf{1}, \mathbf{w}, T, \alpha)x\right)e}_\infty 
 \leq&\sum_{j=1}^4 \sup_{n\geq 0}\norm{e \left(M_n(\mathbf{b},\mathbf{1}, \mathbf{w}, T, \alpha)x_j\right)e}_\infty \\
 \leq&\sum_{j=1}^4 \sup_{n\geq 0}\norm{e_j \left(M_n(\mathbf{b},\mathbf{1}, \mathbf{w}, T, \alpha)x_j\right)e_j}_\infty\\
 \leq & \delta.     
\end{align*}
Therefore, the sequence $\{ M_n(\mathbf{b},\mathbf{1}, \mathbf{w}, T, \alpha)\}_{n\geq 0}$ is $b.u.e.m.$ at $0$ on $L_\Phi(\M)$.

For an arbitrary $\mathbf{b}=\{b_k\}_{k\geq 0}$, decomposing each $b_k$ as a linear combination of four positive elements and arguing as above, we conclude that the sequence $\{M_n(\mathbf{b},\mathbf{1},\mathbf{w},T,\alpha)\}_{n\geq 0}$ is $b.u.e.m.$ at $0$ on $L_\Phi(\M)$. This completes the proof.
\vspace{2mm}

\noindent $($b$)$. Now we deal with the subsequential expression and show that $\{M_n^{\mathbf{k}}(\mathbf{b},\mathbf{1}, \mathbf{w}, T, \alpha)\}_{n\geq 0}$ is $b.u.e.m.$ at $0$ on $L_\Phi(\M)$ for $\alpha\geq 1$. Note that by an argument similar to part $(a)$, it is enough to assume that $b_k\geq 0$ for all $k\geq 0$.  Let us consider the sequence $\mathbf{c} = \{c_j\}_{j \geq 0}$ with $c_j = \chi_{\mathbf{k}}(j)$ for all $j\geq 0$. 
Observe that $\alpha-1\geq 0$ as $\alpha\geq 1$. Now, for any $y\in L_\Phi(\M)_+$, we have
\begin{align}\label{Eq:SubsequenceAveragesInequality}
M_n^{\mathbf{k}}(\mathbf{b},\mathbf{1}, \mathbf{w}, T, \alpha)y
=&\frac{1}{A_n^{\alpha}} \sum_{j=0}^n A_{n-j}^{\alpha-1}w_{k_j}T^{k_j}(b_{k_j} y)\\
\nonumber=&\frac{1}{A_n^{\alpha}} \sum_{j=0}^n \left(\frac{A_{n-j}^{\alpha-1}}{A_{k_n-k_j}^{\alpha-1}}\right) A_{k_n-k_j}^{\alpha-1}w_{k_j}T^{k_j}(b_{k_j} y)\\
\nonumber\leq &\frac{1}{A_n^{\alpha}} \sum_{j=0}^n  A_{k_n-k_j}^{\alpha-1}w_{k_j}T^{k_j}(b_{k_j} y)\qquad (\text{since }\frac{A_{n-j}^{\alpha-1}}{A_{k_n-k_j}^{\alpha-1}}\leq 1)\\
\nonumber= &\frac{1}{A_n^{\alpha}} \sum_{j=0}^{k_n}  A_{k_n-j}^{\alpha-1}w_{j}T^{j}(c_jb_{j} y)\\
\nonumber= &\left(\frac{A_{k_n}^\alpha}{A_n^{\alpha}}\right) \frac{1}{A_{k_n}^\alpha}\sum_{j=0}^{k_n}  A_{k_n-j}^{\alpha-1}w_{j}T^{j}(c_jb_{j} y)\\
\nonumber= &\left(\frac{A_{k_n}^\alpha}{A_n^{\alpha}}\right) M_{k_{n}}(\mathbf{cb},\mathbf{1}, \mathbf{w}, T, \alpha)y,
\end{align}
where $\mathbf{cb} = \{c_jb_j\}_{j \geq 0}$.

Let $K= \sup_{n\geq 1} \frac{k_n}{n}$. Then, $0<K < \infty$ as lower density $d>0$ $($see, \cite[Lemma 40]{Ro94}$)$.
From the part $($a$)$ of this theorem, the sequence $\{M_{k_{n}}(\mathbf{cb},\mathbf{1}, \mathbf{w}, T, \alpha)\}_{n\geq 0}$ is $b.u.e.m.$ at $0$ on $L_\Phi(\M)$.  Therefore, for $\varepsilon, \delta>0$, there exists $\gamma>0$ such that for $\norm{y}_\Phi < \gamma$ with $y\in L_\Phi(\M)_+$, there exists a projection $e \in \mathcal{P}(\M)$ such that
\begin{align}\label{Eq:BUEMForArbitaryx}
\tau(e^\perp)\leq \frac{\varepsilon}{4} \quad \text{and} \quad \sup_{n\geq 0}\norm{e \left(M_{k_n}(\mathbf{cb},\mathbf{1}, \mathbf{w}, T, \alpha)y\right)e}_\infty \leq \frac{\delta}{4(K+1)^{\alpha}\Gamma(\alpha+1)}.
\end{align}
By Eq. \eqref{Eq:AnalphaInequality}, we have
\begin{align}\label{Eq:AknByAn}
\frac{A_{k_n}^\alpha}{A_n^\alpha}   \lesssim_{\alpha}\Gamma(\alpha+1)\left(\frac{k_n+1}{n}\right)^\alpha
\leq  (K+1)^{\alpha}\Gamma(\alpha+1). 
\end{align} 
Therefore, 
\begin{align*}
\sup_{n\geq 0}\norm{e \left(M_n^{\mathbf{k}}(\mathbf{b},\mathbf{1}, \mathbf{w}, T, \alpha)y\right)e}_\infty
\underset{(\text{Eq. }\eqref{Eq:SubsequenceAveragesInequality})}{\leq}& 
\sup_{n\geq 0}\left(\frac{A_{k_n}^\alpha}{A_n^{\alpha}}\right) \norm{e \left(M_{k_{n}}(\mathbf{cb},\mathbf{1}, \mathbf{w}, T, \alpha)y\right)e}_\infty\\
\underset{(\text{Eq. \eqref{Eq:AknByAn}})}{\lesssim_{\alpha}} & 
\big((K+1)^{\alpha}\Gamma(\alpha+1)\big) \sup_{n\geq 0} \norm{e \left(M_{k_{n}}(\mathbf{cb},\mathbf{1}, \mathbf{w}, T, \alpha)y\right)e}_\infty\\   
\leq& \frac{\delta}{4}.
\end{align*}
By an argument similar to that used in part $(a)$ of the proof of this theorem, the above conclusion extends to every $y\in L_\Phi(\M)$ with $\norm{y}_\Phi<\gamma$. Indeed, writing
\begin{align*}
y=y_1-y_2+i(y_3-y_4),
\end{align*}
where $y_j\in L_\Phi(\M)_+$ and $\norm{y_j}_\Phi\leq\norm{y}_\Phi$ for $1\leq j\leq 4$, and applying the preceding argument to each $y_j$, we conclude that the sequence $\{M_n^{\mathbf{k}}(\mathbf{b},\mathbf{1},\mathbf{w},T,\alpha)\}_{n\geq  0}$ is $b.u.e.m.$ at $0$ on $L_\Phi(\M)$. This completes the proof.
\end{proof}


We now establish the pointwise convergence of the weighted ergodic averages, as well as their subsequential counterparts, on a dense subset of the noncommutative Orlicz space.

\begin{proof}[\textbf{Proof of Theorem \ref{Thm:ConvergentForDenseSet}}]
Let $x \in L_1(\M)\cap\M$. We may assume that $x\neq 0$. Suppose $\mathbf{d}=\{d_k\}_{k\geq 0}$ is bounded. Denote 
\begin{align*}
C := \sup_{k \geq 0} \norm{d_k}_\infty < \infty.   
\end{align*}
Let $\varepsilon>0$ be given. Since $\mathbf{b}\in \mathcal{U}_f(1,\alpha)$,  there exists a trigonometric polynomial over $\mathcal{U}_f$, 
\begin{align*}
\psi_1(\cdot) = \sum_{i=1}^m z_i u_i^{(\cdot)},  
\end{align*}
where $u_i\in \mathcal{U}_f$ and $z_i\in \mathbb{C}$, $1\leq i\leq m$, such that
\begin{align}\label{Eq:bkEstimate}
\limsup_{n\rightarrow\infty}\frac{1}{A_{n}^\alpha}\sum_{k=0}^n A_{n-k}^{\alpha-1}\norm{b_k-\psi_1(k)}_\infty<\frac{\varepsilon}{4\norm{x}_{\infty}C}.
\end{align}
Denote $D:=\sum_{i=1}^m \abs{z_i}$. Then, note that
\begin{align}\label{Eq:BoundD}
\sup_{k\geq 1}\norm{\psi_1(k)}_{\infty}
\leq \sum_{i=1}^m \abs{z_i}
= D.
\end{align}
Again, since $\mathbf{d}\in \mathcal{U}_f(1,\alpha)$,  there exists a trigonometric polynomial over $\mathcal{U}_f$, 
\begin{align*}
\psi_2(\cdot) = \sum_{j=1}^l w_j v_j^{(\cdot)},    
\end{align*}
where $v_j\in\mathcal{U}_f$ and $w_j\in \mathbb{C}$, $1\leq j\leq l$, such that
\begin{align}\label{Eq:dkEstimate}
\limsup_{n\rightarrow\infty}\frac{1}{A_{n}^\alpha}\sum_{k=0}^n A_{n-k}^{\alpha-1}\norm{d_k-\psi_2(k)}_\infty<\frac{\varepsilon}{4\norm{x}_\infty D}.
\end{align}

\noindent $($i$)$. Since $x \in L_1(\M)\cap\M$, we write
\begin{align}\label{Eq:AddSubstract}
b_kxd_k-\psi_1(k)x\psi_2(k)
=\left(b_k-\psi_1(k)\right)x d_k
+ \psi_1(k)x \left(d_k-\psi_2(k)\right).
\end{align}
Consequently, we have
\begin{eqnarray*}
& &\norm{M_n\left(\mathbf{b},\mathbf{d},\mathbf{1},T,\alpha\right)x - M_n\left(\psi_1, \psi_2,\mathbf{1},T,\alpha\right)x}_\infty \\
&=&\norm{\frac{1}{A_n^{\alpha}} \sum_{k=0}^n A_{n-k}^{\alpha-1}T^k\left(b_kxd_k-\psi_1(k)x\psi_2(k)\right)}_\infty\\
&\leq &\frac{1}{A_n^{\alpha}} \sum_{k=0}^n A_{n-k}^{\alpha-1}\norm{T^k\left(b_kxd_k-\psi_1(k)x\psi_2(k)\right)}_\infty\\
&\underset{(T\text{ is }DS)}
{\leq} &\frac{1}{A_n^{\alpha}} \sum_{k=0}^n A_{n-k}^{\alpha-1}\norm{b_kxd_k-\psi_1(k)x\psi_2(k)}_\infty\\
&\underset{(\text{Eq. \eqref{Eq:AddSubstract}})}{\leq}& \frac{1}{A_n^{\alpha}} \sum_{k=0}^n A_{n-k}^{\alpha-1}\norm{\left(b_k-\psi_1(k)\right)x d_k}_\infty
+ \frac{1}{A_n^{\alpha}} \sum_{k=0}^n A_{n-k}^{\alpha-1}\norm{\psi_1(k)x \left(d_k-\psi_2(k)\right)}_\infty\\
&\leq& \frac{1}{A_n^{\alpha}} \sum_{k=0}^n A_{n-k}^{\alpha-1}\norm{b_k-\psi_1(k)}_{\infty}\norm{x}_\infty \norm{d_k}_\infty\\
&&\qquad\qquad+ \frac{1}{A_n^{\alpha}} \sum_{k=0}^n A_{n-k}^{\alpha-1}\norm{\psi_1(k)}_{\infty}\norm{x}_{\infty} \norm{d_k-\psi_2(k)}_\infty\\ 
&\leq& \norm{x}_\infty \left(\sup_{k\geq 0}\norm{d_k}_\infty\right)\frac{1}{A_n^{\alpha}} \sum_{k=0}^n A_{n-k}^{\alpha-1}\norm{b_k-\psi_1(k)}_{\infty}\\
& & \qquad\qquad +\left(\sup_{k\geq 0}\norm{\psi_1(k)}_{\infty}\right)\norm{x}_{\infty}\frac{1}{A_n^{\alpha}} \sum_{k=0}^n A_{n-k}^{\alpha-1} \norm{d_k-\psi_2(k)}_\infty\\
&\leq & \left(\norm{x}_\infty C\right) \frac{1}{A_n^{\alpha}} \sum_{k=0}^n A_{n-k}^{\alpha-1}\norm{b_k-\psi_1(k)}_{\infty}
+ \left(D\norm{x}_{\infty}\right)\frac{1}{A_n^{\alpha}} \sum_{k=0}^n A_{n-k}^{\alpha-1} \norm{d_k-\psi_2(k)}_\infty.
\end{eqnarray*}
Taking $\limsup$ on both sides of the above equation and using  Eqs. \eqref{Eq:bkEstimate} and \eqref{Eq:dkEstimate}, we obtain
\begin{align*}
\limsup_{n\rightarrow\infty}\norm{M_n\left(\mathbf{b},\mathbf{d},\mathbf{1},T,\alpha\right)x - M_n\left(\psi_1, \psi_2,\mathbf{1},T,\alpha\right)x}_\infty <\frac{\varepsilon}{2}.   
\end{align*}
Consequently, there exists $n_0\in\mathbb{N}$, such that for all $n\geq n_0$, we have
\begin{align*}
\norm{M_n\left(\mathbf{b},\mathbf{d},\mathbf{1},T,\alpha\right)x - M_n\left(\psi_1, \psi_2,\mathbf{1},T,\alpha\right)x}_\infty <\varepsilon.
\end{align*}
By Lemma \ref{Lem:ConvergenceofTrigonometricPolyNomial}, notice that the sequence $\{M_n(\psi_1, \psi_2,\mathbf{1},T, \alpha)x\}_{n\geq 0}$ is $a.u.$ convergent. Hence, by Lemma \ref{Lem:DifferenceSmallAULemma}, the sequence $\{M_n(\mathbf{b}, \mathbf{d},\mathbf{1},T, \alpha)x\}_{n\geq 0}$ converges  $a.u.$ for every $x \in L_1(\M)\cap\M$. 

The same conclusion holds when $\mathbf{b}=\{b_k\}_{k\geq 0}$ is bounded, by an analogous argument. This proves $($i$)$ in the statement.

\vspace{2mm}

\noindent $($ii$)$. Fix $x \in L_1(\M)\cap\M$. We now prove that the sequence $\{M_n^{\mathbf{k}}(\mathbf{b}, \mathbf{d},\mathbf{1},T, \alpha)x\}_{n\geq 0}$ is $a.u.$ convergent. Notice that sequences $\psi_1=\{\psi_1(k)\}_{k\geq 0}$ and $\psi_2= \{\psi_2(k)\}_{k\geq 0}$ are elements in $ \mathcal{U}_f(1,\alpha)$. Therefore, by Lemma \ref{Lem:ConvergenceofTrigonometricPolyNomial}, the sequence $\{M_n(\psi_1, \psi_2,\mathbf{1},T,\alpha)x\}_{n\geq 0}$ is $a.u.$ convergent. In particular, the subsequence $\{M_{k_n}(\psi_1, \psi_2,\mathbf{1},T,\alpha)x\}_{n\geq 0}$ is $a.u.$ convergent.

\noindent Let
\begin{align*}
B_n(\psi_1, \mathbf{d}, \mathbf{1}, T, \alpha)x = \frac{1}{A_{k_n}^{\alpha}} \sum_{j=0}^n A_{n-j}^{\alpha-1}T^{k_j}(\psi_1(k_j)xd_{k_j}).
\end{align*}
Then, we have
\begin{align}\label{Eq:SubsequenceMinusBn}
\mathrm{I}
=&\norm{M_{k_n}(\psi_1, \psi_2, \mathbf{1},T,\alpha)x - B_n(\psi_1, \mathbf{d}, \mathbf{1}, T, \alpha)x}_\infty \\
\nonumber=& \norm{\frac{1}{A_{k_n}^{\alpha}} \sum_{j=0}^{k_n} A_{k_n-j}^{\alpha-1}T^{j}(\psi_1(j)x\psi_2(j)) - \frac{1}{A_{k_n}^{\alpha}} \sum_{j=0}^n A_{n-j}^{\alpha-1}T^{k_j}(\psi_1(k_j)xd_{k_j})}_\infty \\ 
\nonumber\leq & \mathrm{II}+\mathrm{III},
\end{align}
where
\begin{align*}
\mathrm{II}
=\norm{\frac{1}{A_{k_n}^{\alpha}} \sum_{j=0}^{k_n} A_{k_n-j}^{\alpha-1}T^{j}(\psi_1(j)x\psi_2(j)) - \frac{1}{A_{k_n}^{\alpha}} \sum_{j=0}^{k_n} A_{k_n-j}^{\alpha-1}T^{j}(\psi_1(j)xd_j)}_\infty    
\end{align*}
and 
\begin{align*}
\mathrm{III}
= \norm{\frac{1}{A_{k_n}^{\alpha}} \sum_{j=0}^{k_n} A_{k_n-j}^{\alpha-1}T^{j}(\psi_1(j)xd_j) - \frac{1}{A_{k_n}^{\alpha}} \sum_{j=0}^n A_{n-j}^{\alpha-1}T^{k_j}(\psi_1(k_j)xd_{k_j})}_\infty.     
\end{align*}
Since $T$ is $DS$ operator, we have
\begin{align*}
\mathrm{II}
\leq &\frac{1}{A_{k_n}^{\alpha}} \sum_{j=0}^{k_n} A_{k_n-j}^{\alpha-1}\norm{T^{j}\left(\psi_1(j)x\left(\psi_2(j)-d_j\right)\right)}_\infty \\
\leq &\frac{1}{A_{k_n}^{\alpha}} \sum_{j=0}^{k_n} A_{k_n-j}^{\alpha-1}\norm{\psi_1(j)}_\infty\norm{x}_{\infty}\norm{\psi_2(j)-d_j}_\infty \\
\leq & \left(D \norm{x}_\infty\right) \frac{1}{A_{k_n}^{\alpha}}\sum_{j =0}^{k_n} A_{k_n-j}^{\alpha-1}\norm{d_j - \psi_2(j)}_\infty.
\end{align*}
By taking $\limsup$ both sides of the above inequality, from Eq. \eqref{Eq:dkEstimate},  we have
\begin{align*}
\limsup_{n\rightarrow\infty}\,\mathrm{II}
\leq  \frac{\varepsilon}{4}.
\end{align*}

On the other hand,
\begin{align*}
\mathrm{III}
\leq & \norm{\frac{1}{A_{k_n}^{\alpha}} \sum_{j=0}^{n} \left(A_{k_n-k_j}^{\alpha-1}
-A_{n-j}^{\alpha-1}\right)T^{k_j}(\psi_1(k_j)xd_{k_j})}_\infty +\norm{\frac{1}{A_{k_n}^{\alpha}} {\overset{k_n}{\underset{j\not\in\mathbf{k}}{\underset{j=0}{\sum}}}} A_{k_n-j}^{\alpha-1} T^{j}(\psi_1(j)xd_{j})}_\infty \\
\leq & \mathrm{IV} + \mathrm{V},
\end{align*}
where 
\begin{align*}
\mathrm{IV}
=&\norm{\frac{1}{A_{k_n}^{\alpha}} \sum_{j=0}^{n} \left(A_{k_n-k_j}^{\alpha-1}-A_{n-j}^{\alpha-1}\right)T^{k_j}(\psi_1(k_j)xd_{k_j})}_\infty \\
\leq& \frac{1}{A_{k_n}^{\alpha}} \sum_{j=0}^{n} \left(A_{k_n-k_j}^{\alpha-1}-A_{n-j}^{\alpha-1}\right)\norm{T^{k_j}(\psi_1(k_j)xd_{k_j})}_\infty \\
\underset{(T\text{ is }DS)}{\leq}& \frac{1}{A_{k_n}^{\alpha}}\sum_{j=0}^{n} \left(A_{k_n-k_j}^{\alpha-1} - A_{n-j}^{\alpha-1}\right)\norm{\psi_1(k_j)}\norm{x}_\infty \norm{d_{k_j}}_\infty \\
\leq & \left(\left(\frac{1}{A_{k_n}^{\alpha}}\sum_{j=0}^{n} A_{k_n-k_j}^{\alpha-1} \right)- \frac{A_n^\alpha}{A_{k_n}^\alpha} \right)\left(D\norm{x}_\infty C\right) \\
\underset{(\text{Eq. \eqref{Eq:AnalphaInequality}})}{\lesssim_\alpha}& \left( \frac{1}{(k_n)^\alpha}\sum_{j=0}^n (k_n+1)^{\alpha-1} - \frac{n^\alpha}{\Gamma(\alpha+1)A_{k_n}^\alpha} \right)\left(D\norm{x}_\infty C\right)  \\
=& \left(\frac{(k_n+1)^{\alpha-1}(n+1)}{(k_n)^\alpha} - \frac{(k_n)^\alpha}{\Gamma(\alpha+1)A_{k_n}^\alpha} \frac{n^\alpha}{(k_n)^\alpha} \right)\left(D\norm{x}_\infty C\right) \\
\leq& \left( \frac{(k_n+1)^\alpha}{k_n^\alpha}- \frac{(k_n)^\alpha}{\Gamma(\alpha+1)A_{k_n}^\alpha} \frac{n^\alpha}{(k_n)^\alpha}\right)\left(D\norm{x}_\infty C\right)\qquad(\text{since }n+1\leq k_n+1), 
\end{align*}
and
\begin{align*}
\mathrm{V} 
= \norm{\frac{1}{A_{k_n}^{\alpha}} {\overset{k_n}{\underset{j\not\in\mathbf{k}}{\underset{j=0}{\sum}}}} A_{k_n-j}^{\alpha-1} T^{j}(\psi_1(j)xd_{j})}_\infty 
&\leq \frac{1}{A_{k_n}^{\alpha}} {\overset{k_n}{\underset{j\not\in\mathbf{k}}{\underset{j=0}{\sum}}}} A_{k_n-j}^{\alpha-1} \norm{T^{j}(\psi_1(j)xd_{j})}_\infty \\
&\underset{(T\text{ is }DS)}{\leq}\frac{1}{A_{k_n}^{\alpha}} {\overset{k_n}{\underset{j\not\in\mathbf{k}}{\underset{j=0}{\sum}}}} A_{k_n-j}^{\alpha-1} \norm{\psi_1(j)}_\infty\norm{x}_\infty\norm{d_{j}}_\infty \\
&{\leq}\left(\frac{1}{A_{k_n}^{\alpha}} {\overset{k_n}{\underset{j\not\in\mathbf{k}}{\underset{j=0}{\sum}}}} A_{k_n-j}^{\alpha-1}\right) C\norm{x}_\infty D \\
&\lesssim_\alpha \frac{(k_n+1)^{\alpha-1}(k_n-n)}{(k_n)^{\alpha}} C\norm{x}_\infty D \\
&= \left( 1+ \frac{1}{k_n}\right)^{\alpha-1} \left(1-\frac{n}{k_n}\right)C\norm{x}_\infty D.
\end{align*}
Since $\frac{k_n}{n} \to 1$ as $n \to \infty$ $($because $\mathbf{k}$ has density $1)$, by taking $\limsup$ on both sides of $\mathrm{IV}$ and $\mathrm{V}$, we have
\begin{align*}
\limsup_{n\rightarrow\infty}\,\mathrm{IV} =0
\qquad\text{ and }\qquad
\limsup_{n\rightarrow\infty}\,\mathrm{V} =0.    
\end{align*}
Consequently, by Eq. \eqref{Eq:SubsequenceMinusBn}, we have
\begin{align*}
\limsup_{n\rightarrow\infty} \norm{M_{k_n}(\psi_1, \psi_2, \mathbf{1},T,\alpha)x - B_n(\psi_1, \mathbf{d}, \mathbf{1}, T, \alpha)x}_\infty<\frac{\varepsilon}{4}.    
\end{align*}
Thus, there exists $N \in \mathbb{N}$ such that for all $n \geq N$, we have
\begin{align*}
\norm{M_{k_n}(\psi_1, \psi_2, \mathbf{1},T,\alpha)x - B_n(\psi_1, \mathbf{d}, \mathbf{1}, T, \alpha)x}_\infty<\varepsilon.   
\end{align*}
Therefore, by Lemma \ref{Lem:DifferenceSmallAULemma}, it follows that the sequence $\left\{B_n(\psi_1,\mathbf{d},\mathbf{1},T,\alpha)x \right\}_{n\geq 0}$ is $a.u.$ convergent. 

\vspace{2mm}

Now, it remains to show that $\{M_n^{\mathbf{k}}(\mathbf{b}, \mathbf{d},\mathbf{1},T, \alpha)x\}_{n\geq 0}$ is $a.u.$ convergent. Let us define 
\begin{align*}
B_n(\mathbf{b}, \mathbf{d}, \mathbf{1}, T, \alpha)x 
= \frac{1}{A_{k_n}^{\alpha}} \sum_{j=0}^n A_{n-j}^{\alpha-1}T^{k_j}(b_{k_j}xd_{k_j}).
\end{align*}
Then, 
\begin{align*}
&\norm{B_n(\psi_1, \mathbf{d}, \mathbf{1}, T, \alpha)x - B_n(\mathbf{b}, \mathbf{d}, \mathbf{1}, T, \alpha)x}_\infty \\ 
&= \norm{\frac{1}{A_{k_n}^{\alpha}} \sum_{j=0}^n A_{n-j}^{\alpha-1}T^{k_j}(\psi_1(k_j)xd_{k_j}) - \frac{1}{A_{k_n}^{\alpha}} \sum_{j=0}^n A_{n-j}^{\alpha-1}T^{k_j}(b_{k_j}xd_{k_j}) }_\infty \\ 
&\underset{(T\text{ is }DS)}{\leq} \frac{1}{A_{k_n}^{\alpha}} \sum_{j=0}^n A_{n-j}^{\alpha-1} \norm{b_{k_j} - \psi_1(k_j)}_\infty \norm{x}_\infty C \\
&\leq \frac{1}{A_{k_n}^{\alpha}} \sum_{j=0}^{n} A_{n-j}^{\alpha-1} \norm{b_{k_j} - \psi_1(k_j)}_\infty \norm{x}_\infty C \\
&= \frac{1}{A_{k_n}^{\alpha}} \sum_{j=0}^{n} \frac{A_{n-j}^{\alpha-1}}{A_{k_n-k_j}^{\alpha-1}} {A_{k_n-k_j}^{\alpha-1}}\norm{b_{k_j} - \psi_1(k_j)}_\infty \norm{x}_\infty C \\
&\leq \frac{1}{A_{k_n}^{\alpha}} \sum_{j=0}^{n}  {A_{k_n-k_j}^{\alpha-1}}\norm{b_{k_j} - \psi_1(k_j)}_\infty \norm{x}_\infty C \qquad(\text{since }\frac{A_{n-j}^{\alpha-1}}{A_{k_n-k_j}^{\alpha-1}}\leq 1)\\
&\leq \frac{1}{A_{k_n}^{\alpha}} \sum_{j=0}^{k_n}  {A_{k_n-j}^{\alpha-1}}\norm{b_{j} - \psi_1(j)}_\infty \norm{x}_\infty C.
\end{align*}
By taking $\limsup$ on both sides of the above equation and using Eq. \eqref{Eq:bkEstimate}, we have
\begin{align*}
\limsup_{n\rightarrow\infty}
\norm{B_n(\psi_1, \mathbf{d}, \mathbf{1}, T, \alpha)x - B_n(\mathbf{b}, \mathbf{d}, \mathbf{1}, T, \alpha)x}_\infty 
<\frac{\varepsilon}{4}.
\end{align*}
Consequently, there exists $N \in \mathbb{N}$ such that for all $n \geq N$, we have
\begin{align*}
\norm{B_n(\psi_1, \mathbf{d}, \mathbf{1}, T, \alpha)x - B_n(\mathbf{b}, \mathbf{d}, \mathbf{1}, T, \alpha)x}_\infty 
<\varepsilon.
\end{align*}
Since $\left\{B_n(\psi_1,\mathbf{d},\mathbf{1},T,\alpha)x \right\}_{n\geq 0}$ is $a.u.$ convergent, therefore, by Lemma \ref{Lem:DifferenceSmallAULemma},   $\{ B_n(\mathbf{b}, \mathbf{d}, \mathbf{1}, T, \alpha)x\}_{n\geq 0}$ is $a.u.$ convergent. On the other hand, 
\begin{align*}
M_n^{\mathbf{k}}(\mathbf{b}, \mathbf{d}, \mathbf{1},T, \alpha)x 
=\frac{A_{k_n}^{\alpha}}{A_n^{\alpha}}B_n(\mathbf{b}, \mathbf{d}, \mathbf{1}, T, \alpha)x.    
\end{align*}
By Eq. \eqref{Eq:CalphaCoeffient}, notice that
\begin{align*}
\lim_{n\rightarrow\infty}
\frac{A_{k_n}^{\alpha}}{A_n^{\alpha}}
=\left(\lim_{n\rightarrow\infty}
\frac{A_{k_n}^{\alpha} \Gamma(\alpha+1)}{{k_n}^\alpha}\frac{n^\alpha}{\Gamma(\alpha+1)A_n^{\alpha}} \right
)\left(\lim_{n\rightarrow\infty} \frac{k_n^\alpha}{n^\alpha}\right)
=\lim_{n\rightarrow\infty} \frac{k_n^\alpha}{n^\alpha}
=1,   
\end{align*}
where the last equality follows from the density of $\mathbf{k}$ is $1$. Since the sequence $\{ B_n(\mathbf{b}, \mathbf{d}, \mathbf{1}, T, \alpha)x\}_{n\geq 0}$ is $a.u.$ convergent and $\lim_{n\rightarrow\infty}
\frac{A_{k_n}^{\alpha}}{A_n^{\alpha}}=1$, it follows that 
the sequence $\{M_n^{\mathbf{k}}(\mathbf{b}, \mathbf{d},\mathbf{1},T, \alpha)x\}_{n\geq 0}$ is $a.u.$ convergent. 

The same conclusion holds when $\mathbf{b}=\{b_k\}_{k\geq 0}$ is bounded, by an analogous argument.
This completes the proof.
\end{proof}

\begin{thm}\label{Thm:Closedness}
For $\alpha>0$ and $1< q\leq\infty$, let $\mathbf{b}=\{b_k\}_{k\geq 0}\in W_{q}^\alpha(\mathcal{Z}(\M))$. Let $q^\prime$ denote the conjugate index of $q$, that is, 
$\frac{1}{q}+\frac{1}{q^\prime}=1$. Let $T$ be a $DS^{+}$ operator on $\M$. Then the following statements hold:

\noindent $($a$)$ The set 
\begin{align}\label{Eq:BanachPrinciple1}
\mathcal{A}_{\Phi, 1}=\{x\in L_\Phi(\M): \{M_n(\mathbf{b}, \mathbf{d}, \mathbf{1}, T, \alpha) x\}_{n\geq 0} \text{ converges }b.a.u. \,\,(\text{resp. }a.u.)\}
\end{align}
is closed in $L_\Phi(\M)$ in each of the following cases:
\begin{enumerate}[label=(\roman*)]
\item For $\alpha>0$, $\Phi$ is a $p$-convex Orlicz  function with $\max \{\frac{q^\prime}{\alpha}, q^\prime\}< p<\infty$ $($resp. $\max \{\frac{2q^\prime}{\alpha}, 2q^\prime\}< p<\infty)$. 
\item For $\alpha>1$, $\Phi$ is a $p$-convex Orlicz  function with $p=q^\prime<\infty$ $($resp. $p=2q^\prime<\infty)$.
\end{enumerate}
Moreover, if $\widetilde{\Phi}(t)=\Phi(\sqrt{t})$ is $p$-convex, then $b.a.u.$ convergent in Eq. \eqref{Eq:BanachPrinciple1} will be replaced by $a.u.$ convergent. 

\noindent $(b)$  The set 
\begin{align*}
\mathcal{A}_{\Phi, 2}=\{x\in L_\Phi(\M):\,\{ M_{n}^{\mathbf{k}}(\mathbf{b}, \mathbf{d}, \mathbf{1}, T, \alpha) x\}_{n\geq 0} \text{ converges }b.a.u.\,\, (\text{resp. }a.u.)\}
\end{align*}
is closed in $L_\Phi(\M)$ if 
$\mathbf{k}=\{k_j\}_{j\geq 0}$ is a strictly increasing sequence of natural numbers with lower density $d>0$, and $\alpha$, $p$ and $\Phi$ satisfy the conditions either $(i)$ or $(ii)$ in part $(a)$.
\end{thm}
\begin{proof}
The proof follows directly from \cite[Theorem 2.1]{Li12}, together with Theorem \ref{Thm:buemAt0ForM_nbTx} and the fact that both $\{M_{n}(\mathbf{b}, \mathbf{d}, \mathbf{1}, T, \alpha)\}_{n\geq 0}$ and $\{M_{n}^{\mathbf{k}}(\mathbf{b}, \mathbf{d}, \mathbf{1}, T, \alpha)\}_{n\geq 0}$ are sequences of additive maps from the Banach space $L_\Phi(\M)$ into $L_0(\M)$. Thus, we omit the details. 
\end{proof}

We are now in a position to prove the final result $($Theorem \ref{Thm:AUconvergenceInOrliczSpace}$)$ of this section, which establishes the $b.a.u.$ convergence of the weighted ergodic averages considered in Eqs. \eqref{Eq:ModulatedWeightedAverage} and \eqref{Eq:ModulatedWeightedAverage2} on noncommutative Orlicz spaces.

\begin{proof}[\textbf{Proof of Theorem \ref{Thm:AUconvergenceInOrliczSpace}}]
Suppose $\mathbf{d}=\{d_k\}_{k\geq 0}$ is bounded. Denote 
\begin{align*}
C := \sup_{k \geq 0} \norm{d_k}_\infty < \infty.  
\end{align*}
Let 
\begin{align*}
\mathcal{A}^{\mathbf{k}}_\Phi 
= \{x \in L_{\Phi}(\M) : \{M_n^\mathbf{k}(\mathbf{b}, \mathbf{d}, \mathbf{1}, T, \alpha) x\}_{n\geq 0}  \text{ converges } b.a.u.\}.    
\end{align*}
By Theorem \ref{Thm:ConvergentForDenseSet}, we have $L_1(\mathcal{M}) \cap \mathcal{M} \subseteq \mathcal{A}^{\mathbf{k}}_\Phi$.
Since $L_1(\mathcal{M}) \cap \mathcal{M}$ is dense in $L_{\Phi}(\M)$, by Theorem \ref{Thm:Closedness}, we have $\mathcal{A}^{\mathbf{k}}_\Phi= L_{\Phi}(\M)$.

Fix $x \in L_{\Phi}(\M)$. 
We now show that the $b.a.u.$ limit of the sequence $\{M_n^\mathbf{k}(\mathbf{b}, \mathbf{d}, \mathbf{1}, T, \alpha)x\}_{n\geq 0}$ belongs to $L_\Phi(\M)$. Indeed, since the sequence $\{M_n^\mathbf{k}(\mathbf{b}, \mathbf{d}, \mathbf{1}, T, \alpha)x\}_{n\geq 0}$ converges $b.a.u.$, it also converges in the measure topology.
Since $L_0(\M)$ is complete with respect to the measure topology, there exists $x_0\in L_0(\M)$ such that $M_n^{\mathbf{k}}(\mathbf{b},\mathbf{d},\mathbf{1},T,\alpha)x$ converges to $x_0$ in measure. Moreover, the $b.a.u.$ limit coincides with $x_0$. It remains to show that $x_0\in L_\Phi(\M)$.

Note that 
\begin{align*}
\norm{M_n^\mathbf{k}(\mathbf{b}, \mathbf{d}, \mathbf{1}, T, \alpha)x}_{\Phi}
&\underset{(T\text{ is }DS)}{\leq} \frac{1}{A_n^{\alpha}} \sum_{j=0}^n A_{n-j}^{\alpha -1} \norm{b_{k_j}}_{\infty} \norm{x}_{\Phi}\norm{d_{k_j}}_{\infty} \\
&\leq \left(C \norm{x}_{\Phi}\right) \frac{1}{A_n^{\alpha}} \sum_{j=0}^n A_{n-j}^{\alpha -1} \norm{b_{k_j}}_{\infty} \\
&= \left(C \norm{x}_{\Phi}\right) \frac{1}{A_n^{\alpha}} \sum_{j=0}^n \left(\frac{A_{n-j}^{\alpha -1}}{A_{k_n-k_j}^{\alpha-1}}\right) A_{k_n-k_j}^{\alpha-1}\norm{b_{k_j}}_{\infty} \\
&\leq \left(C \norm{x}_{\Phi}\right) \frac{1}{A_n^{\alpha}} \sum_{j=0}^n  A_{k_n-k_j}^{\alpha-1}\norm{b_{k_j}}_{\infty}
\qquad\left(\text{since }\frac{A_{n-j}^{\alpha -1}}{A_{k_n-k_j}^{\alpha-1}}\leq 1\right)\\
&\leq \left(C \norm{x}_{\Phi}\right) \left( \frac{1}{A_n^{\alpha}} \sum_{j=0}^{k_n} A_{k_n-j}^{\alpha -1} \norm{b_j}_{\infty} \right)\\
&= \left(C \norm{x}_{\Phi}\right) \frac{A_{k_n}^{\alpha}}{A_n^\alpha}\left( \frac{1}{A_{k_n}^{\alpha}} \sum_{j=0}^{k_n} A_{k_n-j}^{\alpha -1} \norm{b_j}_{\infty} \right)
\\
&\underset{(\text{Eq. \eqref{Eq:HolderInequality}})}{\leq} \left(C \norm{x}_{\Phi}\right) \frac{A_{k_n}^{\alpha}}{A_n^\alpha}\left( \frac{1}{A_{k_n}^{\alpha}} \sum_{j=0}^{k_n} A_{k_n-j}^{\alpha -1} \norm{b_j}_{\infty}^q \right)^\frac{1}{q}\left( \frac{1}{A_{k_n}^{\alpha}} \sum_{j=0}^{k_n} A_{k_n-j}^{\alpha -1}  \right)^\frac{1}{q^\prime}\\
&\leq \frac{A_{k_n}^{\alpha}}{A_n^\alpha}\left(C \norm{x}_{\Phi}\right)  \norm{b}_{W_q^\alpha(\M)}\\
&\underset{(\text{Eq. \eqref{Eq:AknByAn}})}{\lesssim_\alpha} C \norm{x}_{\Phi}  \norm{b}_{W_q^\alpha(\M)}.
\end{align*}
Thus, the set $\{M_n^\mathbf{k}(\mathbf{b}, \mathbf{d}, \mathbf{1}, T, \alpha)x:n\geq 0 \}$ lies inside a closed ball of radius $C_\alpha C \norm{x}_{\Phi}\norm{b}_{W_q^\alpha(\M)}$ of $L_\Phi(\M)$ for some $C_\alpha>0$. We know that the unit ball of $L_\Phi(\M)$ is closed under measure topology $($see, \cite[Proposition 5.14]{DDP93}$)$. Hence, 
\begin{align*}
\norm{x_0}_\Phi \leq C_\alpha C\norm{b}_{W_q^\alpha(\M)} \norm{x}_{\Phi},    
\end{align*}
that is, $x_0 \in L_{\Phi}(\M)$.

The same conclusion holds when $\mathbf{b}=\{b_k\}_{k\geq 0}$ is bounded, by an analogous argument. This completes the proof of the theorem.
\end{proof}

\begin{rem}\label{Rem:alphaLessThan1}
We remark that our proof of Theorem \ref{Thm:AUconvergenceInOrliczSpace} $($subsequential part$)$ does not extend to the range $0<\alpha<1$, as the estimates used in our argument are no longer sufficient when $0<\alpha<1$. Whether a corresponding result holds for $0<\alpha<1$ would require a different approach and remains an interesting question for future investigation. As noted in Section \ref{SubSec:Method}, this restriction is not merely a technical convenience. Indeed, the difficulty arises from the fact that $A_n^{\alpha-1}$ is decreasing, rather than increasing, in $n$ when $-1<\alpha-1<0$. Consequently, the key term-wise comparison
\begin{align*}
\frac{A_{n-j}^{\alpha-1}}{A_{k_n-k_j}^{\alpha-1}}\leq 1,
\end{align*}
which is used throughout Section \ref{Sec:MainResults}, is reversed.
\end{rem}

\vspace{0.1in}

\noindent\textbf{Data availability: }Data sharing is not applicable to this article as no datasets were generated or analyzed during the current study.

\vspace{0.1in}

\noindent \textbf{AI Declaration: }The results and proofs in this paper were not generated by AI.

\vspace{0.1in}

\noindent \textbf{Acknowledgment: }
A. Chattopadhyay is supported by the Core Research Grant (CRG), File No: CRG/2023/004826, by the Science and Engineering Research Board(SERB), Department of Science \& Technology (DST).
D. De acknowledges support from the INSPIRE Research Grant (Ref. No. DST/INSPIRE/04/2024/003789), DST, Government of India. D. Shaw gratefully acknowledges the financial support provided by IIT Guwahati, Government of India. Also, all the authors would like to acknowledge the FIST Grant No.: SR/FST/MS-II/2024/183 (C) by the Department of Science \& Technology (DST), Government of India.

\end{document}